\documentclass[11pt]{amsart}

\usepackage{pythagorasHeader}
\usepackage[foot]{amsaddr} 

\newcommand{\notabene}[1]{\footnote{\magenta{#1}}} 
\renewcommand{\notabene}[1]{}

\begin{document}
	
	\title{Pythagoras numbers of orders in biquadratic fields}
	
	\author{Jakub Kr\'asensk\'y$^{1,\ast}$}
	\author{Martin Ra\v{s}ka$^1$}
	\author{Ester Sgallov\'a$^1$}

	\address{$^1$Charles University, Faculty of Mathematics and Physics, Department of Algebra,\newline
		Sokolovsk\'{a}~83, 18600 Praha 8, Czech Republic}
	
	\email{ krasensky@seznam.cz $^{\ast}$, raska.martin@gmail.com, ester.sgallova@centrum.cz}
	\thanks{$^{\ast}$corresponding author}

	

	\begin{abstract}
		We examine the Pythagoras number $\P(\O_K)$ of the ring of integers $\O_K$ in a totally real biquadratic number field $K$. We show that the known upper bound $7$ is attained in a large and natural infinite family of such fields. In contrast, for almost all fields $\BQ{5}{s}$ we prove $\P(\O_K)=5$. Further we show that $5$ is a lower bound for all but seven fields $K$ and $6$ is a lower bound in an asymptotic sense.
		\vspace{0.8em}
		
	\end{abstract}
	
	\setcounter{tocdepth}{2}
	\renewcommand{\contentsname}{}
	\maketitle 
	\thispagestyle{empty}
	\vspace{-4.5em} 
	\tableofcontents
	\clearpage
	
	\section{Introduction}
	
	The study of sums of squares has a long tradition. Already Diophantus was aware that every natural number is a sum of four squares, a fact proven by Lagrange in 1770. Maa{\ss} \cite{Ma} proved that for any totally positive integer of $\Q(\!\sqrt5)$, three squares are sufficient. Siegel \cite{Si} later showed that this is a rather exceptional property: $\Q$ and $\Q(\!\sqrt5)$ are the only totally real number fields where each totally positive algebraic integer is a sum of integral squares.
	
	To remedy this fact, it is customary to consider only those elements of a commutative ring $R$ which can be written as a sum of squares; these elements form a set often denoted by $\sum R^2$. For $\alpha \in \sum R^2$ we define its \emph{length} $\ell(\alpha)$ as the minimal number of squares by which we can represent $\alpha$. For example, $\ell(7)=4$ in $\Z$. The \emph{Pythagoras number} of $R$ is then
	\[
	\P(R) = \sup\{\ell(\alpha) : \alpha \in \textstyle{\sum} R^2 \},
	\]
	i.e.\ the smallest number $n$ such that any sum of squares is in fact a sum of $n$ squares. In this terminology, $\P(\Z)=4$ and $\P(\O_{\Q(\!\sqrt5)})=3$. The Pythagoras number of an order in a number field $K$ is always finite; in fact, if the number field $K$ is not totally real, then $\P(\O)\leq 5$ for any order $\O$, and $\P(\O_K)\leq 4$ for the maximal one (see \cite{Pf}). In the most difficult case of orders of totally real number fields, little is known. Scharlau \cite{Sch} has shown that the Pythagoras number of such an order can be arbitrarily large, but conjectured that it can be bounded by a function depending only on the degree of $K$. This was recently proven by Kala and Yatsyna \cite{KY}; for small degrees $2 \leq d \leq 5$, this bound is $d+3$. (See also Subsection \ref{ss:sqrt5upper} where we prove a stronger version of their result.)
	
	\subsection{Historical overview}
	
	The case of orders in real quadratic fields was completely solved, the main contribution being made by Peters \cite{Pe}. Since we exploit this result repeatedly and since some parts of the proof are either folklore or difficult to find, we collect them separately in Section \ref{se:quadratic}. 
	
	Besides providing an upper bound for the Pythagoras number, Peters also fully characterised which numbers are sums of squares in the orders $\O$ of real quadratic fields -- in other words, he determined the set $\sum \O^2$. It is well known (see e.g.\ \cite{KY}, Lemma 6) that the only \emph{local} conditions for a number $\alpha\in\O_K$ to be a sum of integral squares is that it is totally positive and it is a square modulo $2\O_K$; it seems to be folklore that the same holds for non-maximal orders as well. In a maximal order, every number satisfying these necessary conditions is in fact \emph{locally} a sum of four squares: In non-dyadic completions, three squares suffice, while in a dyadic completion $K_{\mathfrak{p}}$, three squares suffice for each square-mod-2 if and only if $[K_{\mathfrak{p}}:\Q_2]$ is even (\cite{Sch2}, Kapitel 0, Lemma 1). However, these local conditions usually do not suffice to characterise the set $\sum \O^2$: The only real quadratic orders where every totally positive square-mod-2 is \enquote{globally} a sum of squares, are $\Z[\sqrt2]$, $\Z[\sqrt3]$ and $\Z\bigl[\frac{1+\sqrt5}{2}\bigr]$; trivially, the same holds for $\Z$. In Scharlau's dissertation \cite{Sch2}, it is proved that only two other totally real orders of degree at most $4$ share this property, namely the maximal orders of $\Q\bigl(\!\sqrt{\frac{5+\sqrt5}{2}}\bigr)$ and of $\BQ25$. As a corollary of our Theorem \ref{th:mainSqrt5} we shall see that in these two orders, five squares suffice, which is an improvement of the upper bound $7$ provided for quartic orders by \cite{KY} (however, in the latter order we suspect that the correct value is $3$, see Conjecture \ref{co:conjecture}).
	
	This failure of local-global principle is connected with the fact that in totally real fields, the genus of the quadratic form $x^2+y^2+z^2$ rarely consists only of one equivalence class \cite{Pe2}; therefore, it is difficult to prove any global results about sums of squares in these cases (already proving the tight upper bound for $\P(\Z[\sqrt3])$ was tricky, see the discussion below Theorem \ref{th:quadratic}). Nevertheless, there are some recent results for totally real orders. Tinková \cite{Ti} examined the Pythagoras number of orders in simplest cubic fields; one of the present authors \cite{KrCubic} exhibited an exceptional case with very different behaviour. Collinet \cite{Co} and Kala and Yatsyna \cite{KY2} investigated rings of the form $\O_K\bigl[\frac{1}{n}\bigr]$ where $K$ is totally real. For real quadratic orders, B.\ M.\ Kim examined sums of non-vanishing squares with J.\ Y.\ Kim in \cite{KK}, and sums of distinct squares with P.-S.\ Park in \cite{KP}. The papers \cite{KY, KY3} also show an application of the Pythagoras number for the construction of a universal quadratic form. A generalisation of the Pythagoras number, the so-called $g$-invariants, are the subject of recent study as well, see e.g.\ papers by Icaza and Chan \cite{CI, Ic}, Sasaki \cite{SaEng} or Yatsyna and one of the present authors \cite{KrY}. We shall need them in Subsection \ref{ss:sqrt5upper} and provide a definition there.
	
	Studying Pythagoras numbers of fields is easier. For number fields, the local-global principle yields $\P(\,\cdot\,)\leq 4$, so more complicated fields are studied. The case of non-formally real fields is also well understood \cite{Pf}, while the theory for formally real fields is rich -- by Hoffmann \cite{Ho}, every $n\in\N$ appears as the Pythagoras number of such a field -- and still a vivid research area. From many recent papers we mention those of Becher, Grimm, Van Geel, Leep, Hu, Tikhonov and Yanchevskii \cite{BGV, BL, BV, Hu, TVY}.

	\subsection{Our results}
	
	This paper examines the Pythagoras numbers of orders $\O$ in totally real biquadratic fields, i.e.\ in fields of the form $\BQ{p}{q}$ where $p,q>1$ are distinct square-free integers. Since their degree is $4$, the above-mentioned result by Kala and Yatsyna gives $\P(\O) \leq 7$. This, together with the knowledge of the Pythagoras numbers of orders in quadratic fields, were essentially the only tools at our disposal. Our treatment is mostly elementary (the only exception is Subsection \ref{ss:sqrt5upper} where we use some local-global considerations) but some tricks are required to obtain a variety of well-rounded results:
	
	\begin{theorem} \label{th:main7infinitely}
		There exist infinitely many totally real biquadratic fields $K$ with $\P(\O_K)=7$.
	\end{theorem}
	\begin{proof}
		Theorem \ref{th:(B1)P=7} provides a large and natural infinite family of such fields: It suffices to pick two coprime square-free numbers $p \equiv 2$, $q \equiv 3  \pmod{4}$ greater than $7$, and consider $K = \BQ{p}{q}$.
	\end{proof}
	
	The assumption that $K$ is generated by roots of coprime integers is a very common one: For example, multiquadratic fields generated by roots of coprime integers were used both by Scharlau in \cite{Sch} and by Kala and Svoboda in \cite{KS}, in the former case to exhibit orders of arbitrarily large Pythagoras number, in the latter to construct maximal orders without a universal quadratic form of a given rank. Therefore investigating the coprime case was important, and besides the above theorem it also yielded Proposition \ref{pr:coprime}.
	
	Nevertheless, in the present paper we do not restrict to the coprime case, and prove surprisingly strong lower bounds which are valid for each totally real biquadratic field. In particular, the following two statements show that the Pythagoras number of $\O_K$ in most totally real biquadratic fields is not much smaller than the known upper bound $7$:
	
	\begin{theorem} \label{th:main2parts}
		Let $K$ be a totally real biquadratic field. Then:
		\begin{enumerate}
			\item $\P(\O_K)\geq 5$ for all but at most seven cases. \label{it:main(1)}
			\item Fix a square-free positive $n > 7$. Then $\P(\O_K) \geq 6$ for all but finitely many totally real biquadratic fields $K \ni \sqrt n$.
		\end{enumerate}
	\end{theorem}
	\begin{remark}
		The seven fields probably satisfying $\P(\O_K)<5$ are $\BQ23$, $\BQ25$, $\BQ35$, $\BQ27$, $\BQ37$, $\BQ56$, $\BQ57$, see also Conjecture \ref{co:conjecture}.
	\end{remark}
	\begin{proof}
		To prove the first statement is the sole purpose of Section \ref{se:P>=5}; the proof of the second statement is to be found in Section \ref{se:P>=6}.
	\end{proof}
	
	The proofs of Theorems \ref{th:main7infinitely} and \ref{th:main2parts} are based on cleverly finding families of elements with length $7$ ($5$ or $6$, resp.). To handle all fields, it was necessary to divide them into several families and for each family to figure out the best approach. One difficulty is that most biquadratic fields are not generated by square roots of coprime integers, another was the necessity to handle fields containing square roots of $2, 3, 5, 6, 7$ or $13$. One type of fields containing $\sqrt5$ was particularly difficult and required a very different approach, see Subsection \ref{ss:strongAsymptotics}; the developed technique is applicable to other cases, possibly beyond biquadratic fields, and may deserve attention on its own. 
	
	For some computations, especially in Theorem \ref{th:main2parts} (\ref{it:main(1)}), we have written a computer program, starting from scratch. It was primarily used to find elements of needed length for some fields, typically edge cases of our propositions containing square roots of small numbers. On the other hand, it also proved to be useful for finding patterns and establishing conjectures about Pythagoras numbers, which we were usually either able to prove or we formulate them below in Conjecture \ref{co:conjecture}. The used algorithms are described in Subsection \ref{ss:algorithms}, their implementation in Python is available at \url{https://github.com/raskama/number-theory/tree/main/biquadratic}.
	
	\smallskip
	
	Going in the other direction, i.e.\ proving an upper bound, is generally much more difficult. Nevertheless, after generalising the method of Kala and Yatsyna from \cite{KY}, we were able to apply it for quartic fields containing $\sqrt5$ by using a result of Sasaki \cite{SaEng, SaJapan} on $2$-universal quadratic forms.
	
	\begin{theorem}  \label{th:mainSqrt5}
		Let $K$ be a quadratic extension of $\Q(\!\sqrt{5})$ and $\O$ any order in $K$ containing $\frac{1+\sqrt5}{2}$. Then $\P(\O)\leq 5$.
	\end{theorem}
	\begin{proof}
		We prove this as Theorem \ref{th:upperboundSqrt5}.
	\end{proof}
	
	Besides requiring the use of the so-called $g$-invariants, part of the difficulty in proving the above theorem is that the crucial Theorem \ref{th:sasaki} is only available in Sasaki's Japanese paper \cite{SaJapan}. Moreover, the proof contained therein is very brief. For this reason we include a full proof of this result as well.

	Note that by Theorem \ref{th:main2parts} (\ref{it:main(1)}), this improved upper bound is mostly optimal for maximal orders:
	
	\begin{corollary}
		Let $K$ be a totally real biquadratic field containing $\sqrt5$. Then $\P(\O_K)=5$ except possibly when $K = \BQ25$, $\BQ35$, $\BQ56$ or $\BQ57$.
	\end{corollary}
	
	The conclusion that there are infinitely many non-maximal orders with Pythagoras number at most $5$ is also not to be overlooked. One can put it in contrast with the following theorem, which is our main result about non-maximal orders, to see that their Pythagoras number is a rich topic. Although we have a complete proof of the theorem, we did not include it in this paper to shorten the text. It will be published separately.
	
	\begin{theorem} \label{th:mainNonmaximal}
		Let $K$ be a given totally real biquadratic field. There are infinitely many orders $\O \subset K$ satisfying $\P(\O)=7$.
		
		More generally, any totally real biquadratic order $\O_0$ contains an order $\O \subsetneq \O_0$ with $\P(\O)=7$.
	\end{theorem}
	\begin{proof}
		The full proof will be published separately in the article \cite{Kr}. Here we only announce that the orders in question are of the form $\Z + \Z\sqrt{S_0T_0} + \Z\sqrt{M_0T_0} + \Z\sqrt{M_0S_0}$ where $M_0,S_0,T_0 \geq 5$ are any numbers such that $M_0S_0$, $M_0T_0$ and $S_0T_0$ are not squares -- compare this with the notation $m_0, s_0, t_0$ introduced in Subsection \ref{ss:convention}.\notabene{The corresponding element of length seven is $7 + (1+\sqrt{S_0T_0})^2 + (1+\sqrt{M_0T_0})^2 + (1+\sqrt{M_0S_0})^2$.}
	\end{proof}
	
	It is interesting that while the proof of Theorem \ref{th:main7infinitely} worked with square roots of coprime integers, in the case of non-maximal orders we required the order to be generated by square roots of three integers where each two have a common divisor greater than four -- a quite weak condition which nevertheless excludes coprime numbers. This illustrates that while at the first glance the proofs of lower bounds on Pythagoras numbers may all seem the same, they in fact use a variety of ideas.

	Although our conclusions are not as complete as Peters' on quadratic fields (see Section \ref{se:quadratic}), they give a very good idea about the behaviour of Pythagoras numbers of orders in biquadratic fields and thus in totally real fields in general. Regarding maximal orders in biquadratic fields, all the listed results seem to point forwards the following unifying conjecture:\footnote{Based on a preprint of this paper, part (2) of Conjecture \ref{co:conjecture} was proved by He and Hu \cite{HH}.}
	
	\begin{conjecture} \label{co:conjecture}
		Let $K$ be a totally real biquadratic field.
		\begin{enumerate}
			\item If $K$ contains none of $\sqrt2$ and $\sqrt5$, then $\P(\O_K) \geq 6$ holds with finitely many exceptions.\label{it:con1}
			\item If $K$ contains $\sqrt2$ or $\sqrt5$, then $\P(\O_K) \leq 5$.\label{it:con2}
			\item The inequality $\P(\O_K)<5$ holds precisely for the following seven fields:\\ For $K = \BQ23, \BQ25, \BQ35$, where $\P(\O_K)=3$, and\\ for $K =\BQ27, \BQ37, \BQ56, \BQ57$, where $\P(\O_K)=4$. \label{it:con3}
		\end{enumerate}
	\end{conjecture}
	
	We collected evidence for this conjecture in Subsection \ref{ss:conjecture}. Here let us only recall that (\ref{it:con2}) is already proven for $\sqrt5$ in Theorem \ref{th:mainSqrt5}, and that the only missing part of (\ref{it:con3}) are upper bounds for the seven exceptional fields: All necessary lower bounds are contained in Theorem \ref{th:main2parts} (\ref{it:main(1)}) and Proposition \ref{pr:comp_examples}.
	
	Note that while the quadratic fields $F$ generated by $\sqrt2$, $\sqrt3$ and $\sqrt5$ all have $\P(\O_F)=3$ (see Theorem \ref{th:quadratic}), their quadratic extensions behave very differently -- compare the first two statements of the above Conjecture. A possible explanation is that in general, if $K/L$ is a quadratic extension, the value of $\P(\O_K)$ is not directly dependent on $\P(\O_L)$ but rather on the invariant $g_{\O_L}(2)$, see Subsection \ref{ss:sqrt5upper}.

	\subsection{Structure of the paper}
	
	The structure of this paper is as follows: Section \ref{se:prelims} introduces biquadratic fields, repeats their important properties and fixes some notation and important conventions. Section \ref{se:quadratic} concisely presents the known result characterising $\P(\O_F)$ for a quadratic field $F$. In Section \ref{se:lemmata} we collected auxiliary results which are needed throughout the text. Section \ref{se:coprime} deals with biquadratic fields generated by roots of coprime integers, its main but not only result being Theorem \ref{th:main7infinitely}. Section \ref{se:P>=5} contains most of the proof of the first part of Theorem \ref{th:main2parts}, with the hardest case postponed until the next section. Although Section \ref{se:sqrt5} focuses on fields containing $\sqrt5$, its methods for producing both lower and upper bounds are fairly general; in particular, as necessary ingredients, we introduce the $g$-invariants and use them to generalise the upper bound of \cite{KY}, and also provide a full proof of a result known by Sasaki \cite{SaJapan}. In Section \ref{se:P>=6} we proceed to prove the other half of Theorem \ref{th:main2parts} and also to discuss the evidence which supports Conjecture \ref{co:conjecture}. 

	\section{Preliminaries} \label{se:prelims}
	In this section, we explain the used terminology and notation. First we recall a few basic notions from number theory. Then, in Subsection \ref{ss:biquadratic}, we introduce the totally real biquadratic fields, and in Subsection \ref{ss:convention} we fix a very important convention about the meaning of letters $m,s,t$ used \emph{throughout the paper}.
	
	\smallskip
	
	For a number field $K$, we use $\O_K$ to denote its ring of integers. An \emph{order} in $K$ is any subring $\O \subset \O_K$ which has $K$ as its field of fractions; equivalently, $\O$ is both a subring of $K$ and a $\Z$-module of rank $[K : \Q]$. We often call $\O_K$ the \emph{maximal order} in $K$. A number field is \emph{totally real} if the images of all its embeddings into $\C$ are contained in $\R$. For $\alpha \in K$, its \emph{trace} is $\Tr (\alpha)=\sum \sigma_i(\alpha)$, where $ \sigma_i $ runs over all embeddings of $K$ into $\C$. Purely for notational convenience, we will also use the \textit{absolute trace} $\atr$, which is the trace divided by the degree $[K : \Q]$.
	
	Let $K$ be totally real and $\alpha \in K$. If $\sigma_i(\alpha) > 0$ in all embeddings $\sigma_i$, then $\alpha$ is \emph{totally positive}, denoted by $\alpha \succ 0$. If $\alpha=0$ is allowed, we write $\alpha \succcurlyeq 0$ and call it \emph{totally nonnegative}; further, $\alpha \succcurlyeq \beta$ simply means $\alpha - \beta \succcurlyeq 0$. By summing the inequality over all embeddings, it follows that $\Tr \alpha \geq \Tr \beta$ and $\atr \alpha \geq \atr \beta$. It is clear from the definition that any square and thus also any sum of squares is totally nonnegative.
	
	The Pythagoras number $\P(\,\cdot\,)$ of a ring and the length $\ell(\,\cdot\,)$ of an element were already defined in the Introduction. We point out that if $\alpha = x_1^2 + \cdots + x_n^2$, then $\alpha \succcurlyeq x_i^2$ for every $i$ and thus $\atr \alpha \geq \atr x_i^2$.\notabene{A very interesting fact, which is nevertheless never used in this paper, is that the squares-mod-$2$ form a ring (Scharlau denotes them by $R^{(2)}$ in a ring $R$.}

	\subsection{Biquadratic fields} \label{ss:biquadratic}
	
	A \emph{(totally real) biquadratic field} is a number field $K=\BQ{p}{q}$ where $p,q>1$ are two distinct square-free integers. In this paper, we usually omit the words \enquote{totally real}. Let us list a few basic facts about such fields. Further details including the proof of the explicit form of $\O_K$ can be found in \cite{Wi}.
	
	Denote $r = \frac{pq}{\gcd(p,q)^2}$. The numbers $p,q,r$ are the unique square-free rational integers (not counting $1$) which are squares in $K$; one can freely interchange them in the sense that $K=\BQ{p}{q}=\BQ{p}{r}=\BQ{q}{r}$. The only subfields of $K$ are $\Q$, $\Q(\!\sqrt{p})$, $\Q(\!\sqrt{q})$ and $\Q(\!\sqrt{r})$.
	
	The degree $[K :\Q] = 4$, the most natural $\Q$-basis of $K$ being $(1,\sqrt{p},\sqrt{q},\sqrt{r})$. The field $K$ is a Galois extension of $\Q$, with the following $\Q$-automorphisms: If $\alpha = g_0+g_1\sqrt{p}+g_2 \sqrt{q} + g_3 \sqrt{r}$ with $g_i \in \Q$, then
	\[\begin{array}{lll}
		\sigma_1(\alpha)=g_0+g_1\sqrt{p}+g_2 \sqrt{q} + g_3 \sqrt{r},  \\
		\sigma_2(\alpha)=g_0-g_1\sqrt{p}+g_2 \sqrt{q} - g_3 \sqrt{r},  \\
		\sigma_3(\alpha)=g_0+g_1\sqrt{p}-g_2 \sqrt{q} - g_3 \sqrt{r},  \\
		\sigma_4(\alpha)=g_0-g_1\sqrt{p}-g_2 \sqrt{q} + g_3 \sqrt{r}.  \\
	\end{array}\]
	In particular, $\Tr \alpha = 4g_0$ and $\atr \alpha = g_0$, which is the reason why the latter is often more convenient to work with. Note that $\atr (g_0+g_1\sqrt{p}+g_2\sqrt{q}+g_3\sqrt{r})^2 = g_0^2 + pg_1^2+qg_2^2+rg_3^2$. We also see that such a field is indeed totally real.
	
	It is well known that the integral basis of $\O_{\Q(\!\sqrt{n})}$ depends on the value of $n$ modulo $4$. A similar statement holds for biquadratic fields: After possibly interchanging the roles of $p$, $q$ and $r$ (so that $p \equiv r \pmod4$), the field $K$ falls into one of the following five distinct types. For each type we present the appropriate integral basis.
	
	\[\begin{array}{lll}
		\text{(B1)}&\O_K=\Span_{\Z}\Bigl\{1, \sqrt{p}, \sqrt{q}, \frac{\sqrt{p}+\sqrt{r}}2\Bigr\} & \text{if } p\equiv 2,\ q\equiv 3\!\!\!\pmod{4},  \\
		\text{(B2)}&\O_K=\Span_{\Z}\Bigl\{1, \sqrt{p}, \frac{1+\sqrt{q}}2, \frac{\sqrt{p}+\sqrt{r}}{2} \Bigr\} & \text{if } p\equiv 2,\ q\equiv 1\!\!\!\pmod{4}, \\
		\text{(B3)}&\O_K=\Span_{\Z}\Bigl\{1, \sqrt{p}, \frac{1+\sqrt{q}}2, \frac{\sqrt{p}+\sqrt{r}}{2}\Bigr\} & \text{if }  p\equiv 3,\ q\equiv 1\!\!\!\pmod{4},\\
		\text{(B4a)}& \O_K=\Span_{\Z}\Bigl\{1, \frac{1+\sqrt{p}}2,\frac{1+\sqrt{q}}2, \frac{1+\sqrt{p}+\sqrt{q}+\sqrt{r}}{4}\Bigr\} & \text{if } p\equiv q\equiv 1,
		\gcd(p,q)\equiv1\!\!\!\pmod{4},  \\
		\text{(B4b)}& \O_K=\Span_{\Z}\Bigl\{1, \frac{1+\sqrt{p}}2,\frac{1+\sqrt{q}}2, \frac{1-\sqrt{p}+\sqrt{q}+\sqrt{r}}{4}\Bigr\} & \text{if } p\equiv q\equiv 1, \gcd(p,q)\equiv3\!\!\!\pmod{4}.  \\
	\end{array}\]
	
	We shall refer to this distinction often, speaking e.g.\ about \enquote{fields of type (B1)}. Note that the role of $p$ and $r$ is freely interchangeable and in the fields of type (B4), all three letters are interchangeable.
	
	We conclude this section by introducing another point of view on biquadratic fields. Usually, a biquadratic field is given by specifying two of the three values $p,q,r$. However, there is another, more symmetrical way of uniquely describing a biquadratic field: One can specify the values of $p_0, q_0$ and $r_0$, where $p_0=\gcd(q,r)$, $q_0=\gcd(p,r)$, $r_0=\gcd(p,q)$. These three numbers are pairwise coprime and square-free (but one of them is allowed to be $1$); any triple with these properties corresponds to a unique biquadratic field and vice versa. One easily sees that $p=q_0r_0$, $q=p_0r_0$ and $r=p_0q_0$.

	\subsection{An important convention} \label{ss:convention}
	In this subsection, we collect conventions used throughout this paper. Most of them are very natural; however, one of them could easily be overlooked and cause confusion, and is therefore highlighted by a frame. The same convention was already used in \cite{KTZ}.
	
	Unless otherwise stated, $K$ is a totally real biquadratic field, and any biquadratic field is understood to be totally real. If we write $\BQ{p}{q}$, we usually understand $p,q>1$, square-free and distinct, but in important statements, we say it explicitly. When $K=\BQ{p}{q}$, the letter $r$ automatically stands for $r = \frac{pq}{\gcd(p,q)^2}$. We also denote $p_0 = \gcd(q,r)$ etc., see the end of the previous subsection.
	
	In many statements, we use the letters $m,s,t$ rather than $p,q,r$ for the numbers whose square roots generate $K$. They have a precise meaning and are not interchangeable:
	
	\smallskip
	\begin{notation} \label{no:mst}
		\vspace{-0.8em}
		Let $K$ be a biquadratic field. Then the letters $m,s,t$ denote the three square-free rational integers which are squares in $K$, in the following order:
		\[
		1 < m < s < t.
		\]
		We also use $m_0 = \gcd(s,t)$, $s_0=\gcd(m,t)$ and $t_0=\gcd(m,s)$; the above inequality translates into $m_0>s_0>t_0\geq 1$.
		\vspace{-0.2em}
	\end{notation}
	\smallskip
	
	For example, if $K=\BQ23$, then $m=2$, $s=3$, $t=6$. The above convention is \emph{always} in use, so e.g.\ if we write $K = \BQ{7}{s}$, it automatically means (since $7=t$ is impossible) that $m=7 < s < t=7s$; in particular, $K$ is none of the fields $\BQ27$, $\BQ37$, $\BQ57$ and $\BQ67$ and $s$ can be any square-free number satisfying $7 \nmid s$, $s>7$.

	\section{Pythagoras numbers of orders in quadratic fields} \label{se:quadratic}
	
	As we noted in the Introduction, it is difficult to find references for some parts of the characterisation of $\P(\O)$ where $\O$ is an order in a real quadratic field. Therefore we collect them below:
	
	\begin{theorem}[Peters; Cohn and Pall; Dzewas; Kneser; Maa\ss] \label{th:quadratic}
		Let $\O$ be an order in a real quadratic number field. Then
		\[
		\P(\O) =
		\begin{cases}
			3 & \text{for $\O=\Z[\sqrt2]$, $\Z[\sqrt3]$ and $\Z\bigl[\frac{1+\sqrt5}{2}\bigr]$},\\
			4 & \text{for $\O=\Z[\sqrt6]$, $\Z[\sqrt7]$ and the non-maximal order $\Z[\sqrt5]$},\\
			5 & \text{otherwise.}
		\end{cases}
		\]
		The maximal length is attained for example by:
		\begin{itemize}
			\item Length $3$: $1 + \sqrt2^2+(1+\sqrt{2})^2$,\; $2 + (2+\sqrt3)^2$, $2 + \bigl(\frac{1+\sqrt5}{2}\bigr)^2$;\;
			\item Length $4$: $3+(1+\sqrt6)^2$,\; $3+(1+\sqrt7)^2$,\; $3 + (1+\sqrt5)^2$;
			\item Length $5$: $3+\bigl(\frac{1+\sqrt{13}}{2}\bigr)^2+\bigl(1+\frac{1+\sqrt{13}}{2}\bigr)^2$ in $\Z\bigl[\frac{1+\sqrt{13}}{2}\bigr]$; in all the remaining cases $7 + (1+f\sqrt{n})^2$ or $7 + \bigl(f\frac{1+\sqrt{n}}{2}\bigr)^2$ in  $\Z[f\sqrt{n}]$ and $\Z\bigl[f\frac{1+\sqrt{n}}{2}\bigr]$, respectively. Here $n$ is square-free and the integer $f$, called the conductor, can take any positive value.
		\end{itemize}  
	\end{theorem}
	In the following paragraphs, we explain in which reference to find which part of the proof of this theorem.
	
	As far as we know, the upper bound $\P(\,\cdot\,) \leq 5$ is due to Peters \cite{Pe}, and in all but the seven exceptional cases (i.e.\ maximal orders in $\Q(\!\sqrt{n})$, $n=2,3,5,6,7,13$ and the non-maximal order $\Z[\sqrt5]$), he also exhibits the element of length $5$.
	
	The sufficiency of three squares for $n=2$ and $n=5$ stems from the fact that over these fields, the genus of $x^2+y^2+z^2$ contains only one class (Peters cites Dzewas \cite{Dz} but the results are at least partially older, see e.g.\ \cite{Ma}). Showing the same for $\Z[\sqrt3]$ is more involved, since there the genus contains two non-equivalent forms; however, using a nice trick one sees that anything represented by the other form is also represented by three squares. Scharlau, in his dissertation \cite{Sch2}, Kapitel 0, Satz III (iv), attributes this result to \enquote{Kneser, unpublished}.
	
	The sufficiency of four squares in $\Z[\sqrt6]$ and $\Z[\sqrt7]$ is contained in Section 4e of \cite{CP}, together with the curious fact that $12+2\sqrt{13}$ is in fact the \emph{only} (up to conjugation and multiplication by squares of units) element of length $5$ in $\Z\bigl[\frac{1+\sqrt{13}}{2}\bigr]$. This explains why Peters did not find this element (he was clearly unaware of Cohn's and Pall's paper). The sufficiency of four squares in $\Z[\sqrt5]$ is not stated in \cite{CP} explicitly, but it follows by the same method as for $\sqrt6$ and $\sqrt7$, namely using Theorems 2 and 6 and checking elements of small norms.\notabene{It was not that simple! After understanding Theorem 6 it seems to turn out that one must check the totally positive elements of $\Z[2\sqrt5]$ with norm up to $225$ and elements with norms divisible by $16$ up to $400$. Fortunately, by Theorem 2 it suffices to check one element of each norm. But already finding a representative of each admissible norm was quite some work!}
	
	Note that Peters' upper bound $5$ is one instance of the bound given by Kala and Yatsyna in \cite{KY}, which is in turn generalised by our Proposition \ref{pr:upperbound}. Also observe that proving all the listed lower bounds is straightforward -- in Subsection \ref{ss:algorithms} we explain that determining the length of a given element in any fixed order is only a computational task, and in quadratic orders this is easily done with pen and paper, even keeping $f$ and $n$ as parameters in the last two cases. We show a slight variation of this in Observation \ref{ob:betterQuadratic}.
	
	\smallskip
	
	Let us remark that the seven problematic orders for which the element $7 + (1+f\sqrt{n})^2$ or $7 + \bigl(1 + f\frac{1+\sqrt{n}}{2}\bigr)^2$ has length less than $5$, namely the rings of integers in $\Q(\!\sqrt{n})$ for $n=2,3,5,6,7,13$, together with $\Z[\sqrt5]$, are exactly those where $\ell(7)<4$. In particular, $7 = \bigl(\frac{1+\sqrt{13}}{2}\bigr)^2+\bigl(\frac{1-\sqrt{13}}{2}\bigr)^2$.
	
	We shall sometimes exploit parts of the above theorem, namely the information that a particular element has length $5$ (or sometimes $3$ or $4$) in the ring of integers of a quadratic subfield. The knowledge of the quadratic case was the starting point of our examination since it provides a hint on how to construct elements with big length.
	
	\begin{observation} \label{ob:betterQuadratic}
		Peters' choice $7 + \bigl(f\frac{1+\sqrt{n}}{2}\bigr)^2$ for the element of length $5$ is certainly possible, but for non-maximal orders, it is not the simplest one. In this observation, we suggest an alternative. Denote $N = f^2n$; then every quadratic order is either $\Z[\sqrt{N}]$ or $\Z\bigl[\frac{1+\sqrt{N}}{2}\bigr]$ with $N \equiv 1 \pmod4$, where $N$ is no longer square-free, only non-square. Then in $\Z[\sqrt{N}]$ with $N \geq 8$, the element $7 + (1+\sqrt{N})^2$ has length $5$ (this was proven by Peters), while in $\Z\bigl[\frac{1+\sqrt{N}}{2}\bigr]$ with $N \equiv 1 \pmod4$, $N \geq 17$, we claim that
		\[
		\ell\Bigl(7 + \Bigl(\frac{1+\sqrt{N}}{2}\Bigr)^2\Bigr) = 5.
		\]
		We provide the proof of this fact since it is a good and much simpler illustration of what happens later in some of our proofs. Write $\frac{29+N}{4} + \frac{\sqrt{N}}{2} = \sum \bigl(\frac{a_i+b_i\sqrt{N}}{2}\bigr)^2$, $a_i \equiv b_i \pmod2$. This is equivalent to the two equalities $29 + N = \sum a_i^2 + N\sum b_i^2$ and $1 = \sum a_ib_i$. Since changing signs inside of the squares does not affect anything, assume $b_i \geq 0$. Clearly there is a nonzero $b_i$, so $\sum b_i^2 \geq 1$. If $\sum b_i^2 \geq 3$, the first equality yields $29 + N \geq 3N$, a contradiction. If $\sum b_i^2 = 2$, we can assume $b_1=b_2=1$; then $a_1$ and $a_2$ are odd and all other $a_i, b_i$ are even, which contradicts the equality $1 = \sum a_ib_i$.
		
		Therefore $\sum b_i^2 = 1$, say $b_1=1$ and $b_i=0$ otherwise. We have $1 = \sum a_ib_i = a_1$, so the first summand is $\bigl(\frac{1+\sqrt{N}}{2}\bigr)^2$ and all the others are squares of rational integers. Since their sum is $7$, there must be at least four of them, which concludes the proof.
	\end{observation}

	Let us note that while the Pythagoras numbers of real quadratic orders are known, the effort to understand sums of squares in these orders continues: See e.g.\ the already cited papers by B.\ M.\ Kim, J.\ Y.\ Kim, P.-S.\ Park \cite{KK, KP}; one of the present authors recently investigated \cite{Ra} the following question: When does $\sum \O_F^2$ contain all totally positive elements of $m\O_F$ for a given $m \in \N$ and real quadratic field $F$?

	\section{Lemmata and the use of computers} \label{se:lemmata}
	
	The first part of this section contains some general auxiliary results about arithmetics of biquadratic number fields. In the second part we explain how one can use a simple algorithm to prove a lower bound for the Pythagoras number of a given order $\O$. Since we used these self-written computer programs whenever a general \enquote{pen-and-paper} proof failed, in the third part we provide a list of statements proven in this way, to be referenced in later sections.

	\subsection{Some arithmetics of biquadratic number fields}
	Given an algebraic integer $\alpha$ in a quadratic subfield $F$ of our biquadratic field $K$, it is quite possible that $\alpha$ is not a square in $F$ but becomes a square in $K$; a typical example is $2 + \sqrt3 = \bigl(\frac12 (\sqrt{2}+\sqrt{6})\bigr)^2$. (Bear in mind that $\alpha$ is a square in $F$, resp.\ $K$, if and only if it is a square in $\O_F$, resp.\ $\O_K$.) However, the resulting square root of $\alpha$ satisfies a strong condition:
	
	\begin{lemma}\label{le:KTZ}
		Let $K=\BQ{n_1}{n_2}$ with $n_1, n_2>1$ square-free and distinct; put $n_3 = \frac{n_1n_2}{\gcd(n_1,n_2)^2}$.
		
		If $\beta=x+y\sqrt{n_1}+z\sqrt{n_2}+w\sqrt{n_3} \in K$ (with $x,y,z,w \in \Q$) satisfies $\beta^2\in\Q(\!\sqrt{n_1})$, then either $x=y=0$ or $z=w=0$.
	\end{lemma}
	\begin{proof}
		This is Lemma 4.1 of \cite{KTZ}; as far as we know, the original reference is \cite{CLSTZ}.
	\end{proof}

	The fields of type (B4) are the only ones where quarter-integers (i.e.\ numbers from $\Z + \frac14$) occur as coefficients of algebraic integers in the vector space basis. The following lemma shows how such elements behave when squared. Remember that in (B4), the elements $p,q,r$ are freely interchangeable.
	
	\begin{lemma}\label{le:quarterSquares}
		Let $K=\BQ{p}{q}$ be a biquadratic field of type (B4) and $\gamma \in \O_K$ be of the form $\frac{a+b\sqrt{p}+c\sqrt{q}+d\sqrt{r}}{4}$ with $a,b,c,d$ odd. Then $\gamma^2$ is of the same form.
	\end{lemma}
	\begin{proof}
		
		This is just a straightforward computation which uses the fact that $p\equiv q \equiv r \equiv 1 \pmod4$. It suffices to check one of the four coefficients of $\gamma^2$ in the vector space basis $(1,\sqrt{p},\sqrt{q},\sqrt{r})$; if one of them is in $\Z + \frac14$, then the others must be as well, since $\gamma^2$ is an algebraic integer. We choose the first coefficient:
		$\atr\gamma^2 = \frac1{16}(a^2+b^2p+c^2q+d^2r)$. Checking that this is in $\frac{\Z}{4}$ but not in $\frac{\Z}{2}$ is equivalent to proving that $a^2+b^2p+c^2q+d^2r \equiv 4 \pmod8$. And since $\mathrm{odd}^2 \equiv 1 \pmod8$, we actually want
		\[
		1+p+q+r \equiv 4 \pmod{8},
		\]
		which is readily checked.
	\end{proof}

	\subsection{The use of computers} \label{ss:algorithms}
	
	While it is very difficult to find an upper bound for the Pythagoras number of an order and while this paper is mostly about finding lower bounds for \emph{infinite families} of number field orders, the task of finding some lower bound for one given order is easily algorithmically solved. Here we explain the ideas behind such an algorithm and in the next subsection we list some results obtained by using it. The algorithm has a reasonable running time only for elements with small traces, so it becomes hardly usable if we adjoin square roots of large numbers $m,s$; luckily, to fill the gaps in our proofs, we mostly needed to deal with fields where $m,s<100$. Implementation in Python is available at \url{https://github.com/raskama/number-theory/tree/main/biquadratic}.
	
	While we describe the algorithm specifically for full rings of integers, it works for non-maximal orders as well, as long as one knows their integral basis.
	
	Given an $\alpha \in \O_K$, we want to determine whether it is a sum of squares in $\O_K$, and if so, then how many squares are needed. It is in fact possible to find all representations of $\alpha$ as a sum of squares as follows:
	\begin{enumerate}
		\item If $\alpha$ is not totally positive, then it is not a sum of squares.
		\item If $x^2$ occurs in a representation of $\alpha$ as a sum of squares, then $\alpha \succcurlyeq x^2$; in particular, $\atr\alpha \geq \atr x^2$. Therefore we identify all possible summands $x^2$ as follows:
		\begin{enumerate}
			\item Suppose $x = \frac{a+b\sqrt{m}+c\sqrt{s}+d\sqrt{t}}{4}$ and bear in mind that depending on the integral basis, there are congruence conditions imposed on the integers $a,b,c,d$. Find all quadruples which respect these conditions and satisfy
			\[
			\frac{1}{16}(a^2 + b^2m + c^2s + d^2t) \leq \atr\alpha;
			\]
			clearly there are only finitely many of them since $\atr\alpha$ is a given constant.
			\item From the thus obtained list of possible $x^2$, omit those which do not satisfy $\alpha \succcurlyeq x^2$. In this way, the following finite set is obtained:
			\[
			S_{\alpha} = \{x^2 : x \in \O_K, \alpha \succcurlyeq x^2\}.
			\]
		\end{enumerate}
		\item Now it remains to check whether $\alpha$ can be written as a sum of any number of elements of the finite set $S_{\alpha}$. This can be done straightforwardly in many ways, e.g.\ by recursion: Try one element $x^2 \in S_{\alpha}$ and then check whether there are any elements $y^2 \in S_{\alpha}$ satisfying $\alpha - x^2 \succcurlyeq y^2$; if there are, go over them one by one, etc.
	\end{enumerate}
	
	This algorithm is very useful if one has a guess about an element which will require many squares in its representation. If one wants to get an idea about the probable value of $\P(\O_K)$, one can use the following modification, which determines the lengths of all elements up to a given trace:
	\begin{enumerate}
		\item Choose a constant $T$ and find the set
		\[
		S^{(1)} = \{x^2 : x \in \O_K, \atr x^2 \leq T\}.
		\]
		This is done similarly as in the first algorithm.
		\item Compute all sums of two elements of $S^{(1)}$ and reject those whose absolute trace exceeds $T$, thus forming the set
		\[
		S^{(2)} = \{x_1^2+x_2^2 : x_1,x_2 \in \O_K, \atr (x_1^2+x_2^2) \leq T\}.
		\]
		\item By iteratively adding all elements of $S^{(1)}$, compute the sequence of sets
		\[
		S^{(n)} = \{x_1^2 + \cdots + x_n^2 : x_i \in \O_K, \atr (x_1^2 + \cdots + x_n^2) \leq T\}.
		\]
		\item Stop as soon as $S^{(n+1)}=S^{(n)}$. Then $\P(\O_K)\geq n$; in particular, any element of $S^{(n)} \setminus S^{(n-1)}$ has length $n$.
	\end{enumerate}
	
	As $S^{(1)}$ has $O(T^2)$ elements, the time and space complexity of this algorithm is $O(T^{2(n+1)})$, where $n$ is the expected  $\P(\O_K)$ as above. To reduce the space complexity, it is e.g.\ possible to restrict ourselves to finding all sums of $n$ squares with a specific trace (and do this for every trace up to $T$).
	
	Checking whether certain element $\alpha$ is sum of $n$ squares can be done with time complexity $O(T^{2\ceil{n/2}})$, where $T = \atr \alpha$, since one can construct $S^{(\floor{n/2})}$, $S^{(\ceil{n/2})}$ and for every element in $S^{(\floor{n/2})}$ check whether its complementary element is in $S^{(\ceil{n/2})}$.
	
	These time restrictions proved to be sufficient for our needs and also for producing hypotheses about infinite families of $\O_K$. In particular, the results of Subsection \ref{ss:235} were observed after running this algorithm over several hundreds of fields of similar type.

	\subsection{The computed data}
	
	Here we collected those lemmata needed in later proofs which were checked by an algorithm explained in the previous subsection. Of course, in principle they could be verified by a straightforward but very tedious computation with pen and paper.
	
	\begin{lemma}\label{le:comp}
		Let $K = \BQ{m}{s}$ be a totally real biquadratic field with the usual Convention \ref{no:mst} that $m<s<t$ are square-free. To simplify the notation, put
		\[
		\alpha_0 = \begin{cases}
			7 + (1+\sqrt{m})^2 & \text{if $m \not\equiv 1 \pmod 4$},\\
			7 + \Bigl(\frac{1+\sqrt{m}}{2}\Bigr)^2 & \text{if $m \equiv 1 \pmod 4$}.
		\end{cases}
		\]
		Then:
		\begin{enumerate}
			\item If $(m,s)$ is one of $(17,19)$, $(17,21)$, $(17,22)$, $(21,22)$ and $(21,23)$, then $\ell(\alpha_0)=5$. \label{it:mIs1} 
			\item If $(m,s)$ is one of $(10,11)$, $(10,17)$, $(10,19)$, $(11,14)$, $(11,17)$, $(11,21)$; $(14,15)$, $(14,17)$, $(14,19)$, $(15,17)$, $(15,21)$, $(15,22)$, $(19,21)$, $(19,22)$, $(22,23)$, $(23,26)$ -- i.e.\ if $K$ lies in $\K_1 \cup \K_2 \cup \K_3$ in the notation of Subsection \ref{ss:P>=5mEquiv23} and $(m,s) \neq (10,13)$, $(11,13)$, $(30,35)$ -- then $\ell(\alpha_0)=5$.  \label{it:mNot1Gen}
			\item If $(m,s)$ is $(10,13)$ or $(11,13)$, then $3+(\frac{1+\sqrt{13}}{2})^2+(1+\frac{1+\sqrt{13}}{2})^2$ has length $5$; if $(m,s) =  (30,35)$, the same holds for $7 + (1+\sqrt{42})^2$. \label{it:threeexceptions}
			\item If $(m,s)$ is one of $(22,26)$, $(26,30)$, $(19,23)$, then $15 + (1+\sqrt{m})^2$ has length $5$. \label{it:mNot1SpecGood15} 
			\item If $(m,s)$ is one of $(30,38)$, $(34,42)$, $(38,46)$, $(31,39)$, $(35,43)$, $(39,47)$, $(43,51)$, $(47,55)$; $(34,46)$, $(31,43)$, $(35,47)$, $(39,51)$, $(43,55)$, then $28 + (1+\sqrt{m})^2$ has length~$5$. \label{it:mNot1SpecGood28} 
			\item If $(m,s)$ is one of $(10, 14)$, $(11,15)$, $(15,19)$; $(14,22)$, $(22,30)$, $(26,34)$, $(11,19)$, $(15,23)$, $(23,31)$; $(10,22)$, $(14,26)$, $(22,34)$, $(26,38)$, $(11,23)$, $(19,31)$, $(23,35)$,
			then $7+\bigl(\frac{\sqrt{m}+\sqrt{s}}{2} \bigr)^2$ has length $5$.  \label{it:mNot1SpecBad}
			\item If $m=6$ and $s$ is one of $13$, $17$, $29$, $37$, $41$; $7$, $11$, $19$; $26$, $34$; $14$, $22$, $38$, $46$, $62$, $70$, $86$, then the $\alpha$ defined in Proposition \ref{pr:sqrt6} has length $5$.  \label{it:sqrt6}
			\item If $m=6$ and $s=10$, then $3+\bigl(\sqrt{6} -\frac{\sqrt{6}+\sqrt{10}}{2}\bigr)^2+\bigl(2+\frac{\sqrt{6}+\sqrt{10}}{2}\bigr)^2$ has length $5$.  \label{it:sqrt6spec}
			\item If $m=7$ and $s$ is one of $13$, $17$, $29$, $33$, $37$, $41$; $10$; $15$, $19$, $23$, $31$, $39$, $43$, $47$, $51$, then the $\alpha$ defined in Proposition \ref{pr:sqrt7} has length $5$. \label{it:sqrt7}
			\item If $m=7$ and $s=11$, then $3+\bigl(\frac{\sqrt{7}+\sqrt{11}}{2}\bigr)^2+\bigl(1+\frac{\sqrt{7}+\sqrt{11}}{2}\bigr)^2$ has length $5$.\notabene{By the way, an analogy of this holds for all the fields of same type contained in the previous statement.} 
			\label{it:sqrt7spec}
			\item If $m=2$ and $s$ is one of $11$, $15$, $17$, $21$, $29$, then the $\alpha$ defined in Proposition \ref{pr:sqrt2} has length $5$. \label{it:sqrt2}
			\item If $m=2$ and $s=13$, then $3+(1+\sqrt{2})^2+\Bigl(\sqrt{2}+\frac{1+\sqrt{13}}{2}\Bigr)^2+\Bigl(2+\frac{1+\sqrt{13}}{2}+\frac{\sqrt{2}+\sqrt{26}}{2}\Bigr)^2$ has length $5$. \label{it:sqrt2spec}
			\item If $m=3$ and $s$ is one of $13$, $17$, $29$, $37$, $41$; $10$; $11$, $19$, $23$, $31$, $35$, $43$, then the $\alpha$ defined in Proposition \ref{pr:sqrt3} has length $5$.  \label{it:sqrt3}
			\item If $m=5$ and $s = 13$ or $s=17$, then the $\alpha$ defined in Proposition \ref{pr:sqrt5sIs1} has length $5$. \label{it:sqrt5sIs1}
			\item If $m=5$ and $s\not\equiv 1 \pmod4$, $7 < s \leq 3253$, then the $\alpha$ defined in Proposition \ref{pr:sqrt5sNot1} has length $5$. \label{it:sqrt5sNot1}
			\item If $m=13$ and $s$ is one of $17$, $21$, $29$, $33$, $37$, $41$, then the $\alpha$ defined in Proposition \ref{pr:sqrt13P>=6} has length $6$. \label{it:sqrt13P>=6}
			\item In fields $\BQ{13}{n}$ with $n=6,7,10,11$, the element $12+2\sqrt{13} + (1+\sqrt{n})^2$ has length $6$. \label{it:sqrt13=s}
		\end{enumerate}
	\end{lemma}
	\begin{proof}
		As we pointed out, proving that an element $\alpha$ is not a sum of $n$ squares in a given field is just a straightforward computational task. Therefore, we proved this lemma by implementing the above-described algorithm in Python.
		
		One small exception is part (\ref{it:sqrt5sNot1}), but only for technical reasons: To reduce computation time, we verified the statement by computer only for $7 < s \leq 499$, while the remaining cases were handled by refining the proof of Proposition \ref{pr:ourStrongAsymptotics} -- see Observation \ref{ob:sqrt5Improvement}.
	\end{proof}
	
	During our computations, only seven fields stood out as possible candidates for $\P(\O_K) < 5$ (for all other fields, we were able to find an element of length $5$ with a relatively small trace). Especially, for $K = \BQ{2}{3},\BQ{2}{5},\BQ{3}{5}$ we suspect $\P(\O_K)=3$ and for $K =\BQ27, \BQ37, \BQ56, \BQ57$ we suspect $\P(\O_K)=4$. 
	
	The following proposition collects our results about these seven fields: We present an element of the suspected maximal length as well as bounds up to which the computation was completed. The example elements $\alpha_0$ are chosen with the lowest possible trace (many elements with the desired length can be found).
	
	\begin{proposition}\label{pr:comp_examples}
		If $K$, $l$ and $\alpha_0$ is one of the triples specified below then $\ell(\alpha_0) = l$ and $\ell(\alpha) \leq l$ for all $\alpha \in \sum \O_K^2$ such that $\Tr \alpha \leq 500$.
		
		\begin{enumerate}
			\item $K= \BQ23$, $l = 3$, $\alpha_0 = 6+\sqrt{2}+\sqrt{6}$;
			\item $K= \BQ25$, $l = 3$, $\alpha_0 = 6+\sqrt{5}$;
			\item $K= \BQ35$, $l = 3$, $\alpha_0 = 3+\frac{1+\sqrt{5}}{2}$;
			\item $K= \BQ27$, $l = 4$, $\alpha_0 = 10+2\sqrt{2}+\sqrt{7}$;
			\item $K= \BQ37$, $l = 4$, $\alpha_0 = 8+\sqrt{3}+\sqrt{7}$;
			\item $K= \BQ56$, $l = 4$, $\alpha_0 = 10+2\sqrt{6}+\frac{1+\sqrt{5}}{2}$;
			\item $K= \BQ57$, $l = 4$, $\alpha_0 = 11+2\sqrt{7}+\frac{1+\sqrt{5}}{2}$.
		\end{enumerate}
	\end{proposition}
	
	Note that a nontrivial upper bound shall be proven only for fields containing $\sqrt5$, and in general, obtaining an upper bound even for a single field is difficult: Therefore, although we are convinced that e.g.\ $\P(\O_K)=3$ for $K=\BQ23$, we know only $3 \leq\P(\O_K)\leq 7$.
	
	\smallskip
	
	The following proposition collects some of our computed data for fields containing $\sqrt3$. Based on it, we conjecture that there are only finitely many fields $K=\BQ{3}{s}$ with $\P(\O_K)\leq 5$; this is in sharp contrast with the situation in fields containing $\sqrt2$ or $\sqrt5$ (see the conjecture in Subsection \ref{ss:conjecture} and compare this with the situation in quadratic fields, described in Section \ref{se:quadratic}).
	
	\begin{proposition}\label{pr:sqrt3shock}
		For $K=\BQ{3}{s}$ where $s=17$, $s=22$ or $26 \leq s \leq 55$, we have $\P(\O_K) \geq 6$. If $s$ is odd, the same holds also for $55 < s \leq 79$.
		
		In all these cases, there exists $\alpha \in \O_K$ with $\ell(\alpha)=6$ and $\Tr(\alpha)\leq 1000$.  For example, the following elements have length $6$:
		\[
		87 + 18\sqrt3 - 20\sqrt{17} - 4\sqrt{51}; \quad 150 + 45\sqrt3 - 14\sqrt{26}; \quad \frac12(199 - 50\sqrt3 + 14\sqrt{31} - 19\sqrt{93}).
		\]
	\end{proposition}
	
	Although our program showed that for $s = 58, 62, 70$ or $74$, elements with $\Tr(\alpha)\leq 1100$ have length at most five, we believe that $\P(\O_K)\geq 6$ holds even in those fields (and in fact for all $K \ni \sqrt3$ with $s \geq 26$). A hint towards this is also the fact that for $s=46$, all elements with $\Tr(\alpha)\leq 900$ have $\ell(\alpha)\leq 5$. The computer we used is not a particularly strong one, since numerical experiments were not our priority.

	\section{Results for fields with coprime \texorpdfstring{$m, s$}{m, s}} \label{se:coprime}

	In this section we examine the specific fields $K = \BQ{m}{s}$ where $m$ and $s$ are coprime, i.e.\ $t = ms$. They are easier to work with than the general case (and therefore also often occur in works of other authors); in particular, if we restrict to the fields of type (B1), we were able to prove that the Pythagoras number of $\O_K$ is $7$ (with a few exceptions); see Theorem \ref{th:(B1)P=7}. This proves one of the main results of this paper, Theorem \ref{th:main7infinitely}. The other results of this section can be put neatly together as follows:
	
	\begin{proposition} \label{pr:coprime}
		Let $p,q \notin\{2,3,5,6,7,13\}$ be two coprime square-free positive integers which are not both congruent to $3$ modulo $4$.\notabene{I didn't try that much, but I've managed to solve the problematic case $m\equiv s \equiv 3$ only if $s>32+3m$. In such a case, the length of $7 + (1+\sqrt{m})^2 + \bigl(\frac{\sqrt{m}+\sqrt{s}}{2}\bigr)^2$ is $6$. It's quite possible that the same element works in all other cases as well!} Then $\P(\O_K)\geq 6$ for $K=\BQ{p}{q}$.
	\end{proposition}
	\begin{proof}
		This is just a combination of Propositions \ref{pr:B1coprime}, \ref{pr:B23comprime} and \ref{pr:B4coprime}.
	\end{proof}
	
	As a side remark, the inequalities for $p,q$ in all the following propositions are optimal in the sense that if they are violated, then the given $\alpha$ has length less than $6$. To see this, observe that the inequalities only exclude fields containing square roots of $2$, $3$, $5$, $6$, $7$ or $13$, and in these fields $\ell(7)<4$, see Section \ref{se:quadratic}. On the other hand, at least for fields containing $\sqrt{13}$ another element of length $6$ can be found, see Proposition \ref{pr:sqrt13P>=6}.
	
	We start with a result for fields of type (B1); besides being one third of the above proposition, it will also serve as a lemma for the upcoming Theorem \ref{th:(B1)P=7}.
	
	\begin{proposition}[Type (B1)] \label{pr:B1coprime}
		Let $K=\BQ{p}{q}$ where $p\equiv 2$ and $q\equiv3 \pmod{4}$ are coprime square-free integers such that $p,q\geq 10$. Then $\P(\O_K)\geq 6$.
		
		In particular, the following element has length $6$:
		\[
		\alpha = 7+(1+\sqrt{p})^2+(1+\sqrt{q})^2.
		\]
	\end{proposition}
	\begin{proof}
		In this case, the integral basis is given by $\O_K=\Span_{\Z}\bigl\{1, \sqrt{p}, \sqrt{q}, \frac{\sqrt{p}+\sqrt{pq}}2\bigr\}$. Suppose
		\[
		\alpha = \sum \Bigl(a_i + b_i\sqrt{p} + c_i\sqrt{q} + d_i\frac{\sqrt{p}+\sqrt{pq}}2\Bigr)^2.
		\]
		This means
		\begin{align*}
			9+p+q+2\sqrt{p}+2\sqrt{q}  = \sum\limits \Bigl(&\bigl(a_i^2+b_i^2p+c_i^2q+d_i^2\frac{p+pq}{4}+b_id_ip\bigr)\\
			&+ (2a_ib_i+a_id_i+c_id_iq)\sqrt{p}
			+ \bigl(2a_ic_i+b_id_ip+d_i^2\frac{p}{2}\bigr)\sqrt{q}\\
			&+ \left(2b_ic_i+c_id_i+a_id_i\right)\sqrt{pq}\Bigr).
		\end{align*}
		
		By comparing coefficients, one gets the following conditions:
		
		\begin{align}
			9+p+q & = \sum a_i^2+q\sum c_i^2+p\sum\biggl(\!\Bigl(b_i+\frac{d_i}{2}\Bigr)^2+q\frac{d_i^2}{4}\biggr), \label{eq:B1c:1}\\
			2 & = 2\sum a_ib_i+\sum a_id_i+q\sum c_id_i, \label{eq:B1c:2} \\
			2 & = 2\sum a_ic_i+p\sum b_id_i+\frac{p}{2}\sum d_i^2, \label{eq:B1c:3}\\
			0 & = 2\sum b_ic_i+\sum c_id_i+\sum a_id_i. \label{eq:B1c:4}
		\end{align}
		Without loss of generality $a_i\geq0$ and we will prove the following statements:
		\begin{enumerate}[(i)]
			\item There is a nonzero $b_i$ and a nonzero $c_j$.
			\item All the coefficients $d_i$ are zero.
			\item There is exactly one nonzero coefficient $|b_i|=1$ and one nonzero coefficient $|c_j|=1$.
			\item If $c_i\neq 0$, then $b_i=0$ and vice versa.
			\item If $c_i$ is nonzero, then $a_i=c_i=1$, the same for $b_i$.
		\end{enumerate}
		The proofs follow:
		
		\textbf{(i)} Suppose $b_i = 0$ for all $i$.
		Then from (\ref{eq:B1c:4}) we get $\sum c_id_i = -\sum a_id_i$. \\
		Together with (\ref{eq:B1c:2}), this yields
		\[
		2 = q\sum c_id_i - \sum c_id_i = (q-1) \sum c_id_i,
		\]
		which is impossible as $q\geq 10$ and $\sum c_id_i$ is rational integer.
		
		If $c_i = 0$ for all $i$, we proceed similarly: From (\ref{eq:B1c:3}) follows 
		$2 = \frac{p}{2}\sum (2b_id_i+d_i^2)$, 
		which is again impossible as $p\geq 10$ and $\sum 2b_id_i+d_i^2$ is rational integer.

		\textbf{(ii)} Suppose there is some $|d_i|> 0$. Since there is also some nonzero $c_i$, from (\ref{eq:B1c:1}) we get 
		$9+p+q \geq q+ \frac{pq}{4}$.
		This is not possible for $p,q\geq 10$.
		
		\textbf{(iii)}
		We already know that there is at least one nonzero $c_j$ and $b_i$. On the other hand, equation (\ref{eq:B1c:1}) now yields 
		\[
		9+p+q \geq q\sum c_i^2+p\sum b_i^2.
		\]
		If $\sum c_i^2>1$ or $\sum b_i^2>1$ this is clearly impossible as $p,q\geq 10$.
		
		\textbf{(iv)} This is an easy consequence of  $0 = 2\sum b_ic_i$ (\ref{eq:B1c:4}) and the previous statement.

		\textbf{(v)} These are again direct consequences of $1 = \sum a_ib_i$  (\ref{eq:B1c:2}) and $1 = \sum a_ic_i$  (\ref{eq:B1c:3}) and our assumption that $a_i\geq 0$. 
		
		From the previous statements, it can be seen that the only way how to write $\alpha$ as a sum of squares is $\alpha = (1+\sqrt{p})^2+(1+\sqrt{q})^2+\sum x_i^2$, where $x_i \in \Z$. This concludes the proof as $7$ is not a sum of three squares over rational integers. 
	\end{proof}

	Exploiting the previous proposition, we are able to prove one of our main results. Observe that it covers \emph{all} fields with coprime $m,s \geq 10$ of type (B1), since if $m,s$ are coprime, they cannot both be $2$ modulo $4$.
	
	\begin{theorem}[Type (B1)] \label{th:(B1)P=7}
		Let $K=\BQ{p}{q}$ where $p\equiv 2$ and $q\equiv3 \pmod{4}$ are coprime square-free integers such that $p,q\geq 10$. Then $\P(\O_K)=7$.
		
		In particular, the following element has length $7$:
		\[
		\alpha = 7 +(1+\sqrt{p})^2 +(1+\sqrt{q})^2 + \Bigl(\frac{\sqrt{p}+\sqrt{pq}}{2}\Bigr)^2.
		\]
	\end{theorem}
	\begin{proof}
		It suffices to prove that $\ell(\alpha)=7$; the rest is clear since $7$ is an upper bound on the Pythagoras number for every order of degree $4$. For a field of type (B1), we have $\O_K=\Span_{\Z}\bigl\{1, \sqrt{p}, \sqrt{q}, \frac{\sqrt{p}+\sqrt{pq}}2\bigr\}$. Suppose that $\alpha =
		9+p+q+\frac{p+pq}{4}+2\sqrt{p}+\left(2+\frac{p}{2}\right)\sqrt{q}$ is a sum of an arbitrary number of squares.
		
		Analogously to the proof of Proposition \ref{pr:B1coprime}, we denote the squares as $\bigl(a_i + b_i\sqrt{p} + c_i\sqrt{q} + d_i\frac{\sqrt{p}+\sqrt{pq}}2\bigr)^2$. By comparing coefficients, one gets the following conditions:
		
		\begin{align}
			9+p+q+\frac{p+pq}{4} & = \sum  a_i^2+q\sum c_i^2+p\sum \biggl(\!\Bigl(b_i+\frac{d_i}{2}\Bigr)^2+\frac{d_i^2q}{4}\biggr),  \label{eq:B1_7:1}\\
			2 & = 2\sum a_ib_i+\sum a_id_i+q\sum c_id_i, \label{eq:B1_7:2}\\
			2+\frac{p}{2} & = 2\sum a_ic_i+p\sum b_id_i+\frac{p}{2}\sum d_i^2, \label{eq:B1_7:3}\\
			0 & = 2\sum b_ic_i+\sum c_id_i+\sum a_id_i. \label{eq:B1_7:4}
		\end{align}
		Without loss of generality all $d_i\geq0$ and we will step by step prove the following lemmata:
		\begin{enumerate}[(i)]
			\item There exists a nonzero $b_i$ and a nonzero $c_j$.
			\item There is exactly one nonzero $d_i$, let us say $d_1$. This $d_1 = 1$ and $b_1 \in \{-1,0\}$.
			\item There exists a nonzero $b_i$ for $i>1$.
			\item There is exactly one $j$ such that $c_j$ is nonzero, and exactly one $k>1$ such that $b_k$ is nonzero. For these indices, $|c_j|=1$ and $|b_k|=1$. Also $\sum a_i^2 = 9$.
			\item $a_1 = b_1 = c_1 = 0$. Thus, the first summand is $\bigl(\frac{\sqrt{p}+\sqrt{pq}}{2}\bigr)^2$.
		\end{enumerate}
		The last lemma implies that $\ell(\alpha)=7$ since $7+(1+\sqrt{p})^2+(1+\sqrt{q})^2$ has length $6$ by Proposition \ref{pr:B1coprime}. 
		
		Now let us present the proofs of the lemmata:
		
		\textbf{(i)} This can be proven in the same way as the first part of the proof of \ref{pr:B1coprime}.
		
		\textbf{(ii)} If $d_i^2\geq 4$ for some $i$ or there are at least two $d_i = 1$, the RHS of (\ref{eq:B1_7:1}) is at least $q+\frac{pq}{2}$ (since there is a nonzero $c_i$); however, $q+\frac{pq}{2}>9+p+q+\frac{p+pq}{4}$ for $p,q \geq 10$.
		On the other hand, comparing parities of LHS and RHS of (\ref{eq:B1_7:3}), there must exist at least one odd $d_i$, implying there is exactly one nonzero $d_i$, without loss of generality $d_1 = 1$.
		
		Now if $b_1 \notin \{-1,0\}$, then $p\bigl( b_1+\frac{d_1}{2}\bigr)^2\geq \frac{9}{4}p$, which, taken together with the other known estimates (nonzero $c_i$ and $d_1=1$), makes (\ref{eq:B1_7:1}) impossible for $p\geq 10$.
		
		\textbf{(iii)} If $b_1=0$, it is clear from \textbf{(i)}, so assume $b_1=-1$ and $b_i=0$ for $i>1$.
		
		Then from (\ref{eq:B1_7:2}) we get $2 = -2a_1+a_1+c_1q$, from (\ref{eq:B1_7:4}) also $0 = -2c_1+c_1+a_1$. By adding these two equalities, we get $2 = c_1(q-1)$, which is impossible.
		
		\textbf{(iv)} This almost immediately follows from earlier results and (\ref{eq:B1_7:1}): With our knowledge, the equality can be rewritten as
		\[
		9+p+q+\frac{p+pq}{4} = \sum  a_i^2 + q\sum c_i^2 + p \frac{1+q}{4} + p\sum_{i \geq 2} b_i^2.
		\]
		Since there is at least one nonzero $c_i$, we know $\sum c_i^2 \geq 1$, and similarly  $\sum_{i \geq 2} b_i^2 \geq 1$. Thus the RHS is at least $p+q+\frac{pq+p}{4}$, which is LHS minus $9$. From this it is clear that both aforementioned sums must be exactly $1$ (implying that exactly one summand in each of them is $1$) and $\sum a_i^2 =9$.
		
		\textbf{(v)}
		From (\ref{eq:B1_7:3}) we  get $2-pb_1 = \sum 2a_ic_i$. This is impossible for $b_1\neq 0$, as one easily deduces from $p\geq 10$ and the fact that there is exactly one nonzero $|c_i|=1$ and that $\sum a_i^2 =9$.
		
		If $c_1\neq 0$, from (\ref{eq:B1_7:2}) we get $2-c_1q - a_1 =\sum_{i>1} 2a_ib_i$.
		Since $\sum a_i^2 = 9$ and $q\geq 10$, this is again impossible as there is only one nonzero $b_i$ for $i>1$, equal to $\pm1$. Thus $c_1=0$.\notabene{For $q=11$ it might seem possible at the first glance, but only if there are two $a_i$ with $|a_i|=3$. That is clearly nonsense.}
		
		From (\ref{eq:B1_7:4}) we get $a_1 = -\sum_{i>1}2b_ic_i$, so the only other possibility besides $a_1=0$ is $a_1 = \pm2$.
		
		If $a_1=-2$, equations (\ref{eq:B1_7:2}), (\ref{eq:B1_7:3}), (\ref{eq:B1_7:4}) read as $2 = \sum_{i>1}a_ib_i$, $1= \sum_{i>1} a_ic_i$, $1 = \sum_{i>1} b_ic_i$. The last one shows that $b_i \neq 0$ and $c_j\neq 0$ must happen for $i=j$, but then the first two equations imply different values of $a_i$, which is a contradiction.
		
		If $a_1=2$, we get $0 = \sum_{i>1}a_ib_i$, $1= \sum_{i>1} a_ic_i$, $-1 = \sum_{i>1} b_ic_i$, which leads to the same contradiction.
	\end{proof}
	
	The following propositions handle the fields of type (B2,3) and (B4). They can be proved using the exact same five steps as in the proof of Proposition $\ref{pr:B1coprime}$, each of these steps analogously requiring either straightforward inequalities concerning trace or a simple parity argument. As all these ideas were already illustrated, we omit these proofs.
	
	One thing to mention is that instead of using the integral basis explicitly, one can work with elements of the form $\frac{a_i}{4}+\frac{b_i}{4}\sqrt{p}+\frac{c_i}{4}\sqrt{q}+\frac{d_i}{4}\sqrt{pq}$ and impose congruence relations upon the coefficients. The structure of the proofs would remain unchanged and it can perhaps be clearer, especially for fields of type (B4). This other approach is often used throughout the rest of the paper.
	
	\begin{proposition}[Type (B2,3)] \label{pr:B23comprime}
		Let $K=\BQ{p}{q}$ where $p\equiv 2,3$ and $q\equiv 1 \pmod{4}$ are coprime square-free integers such that $p\geq 10$, $q\geq 17$. Then $\P(\O_K)\geq 6$.
		
		In particular, the following element has length $6$:
		\[
		\alpha = 7+(1+\sqrt{p})^2+\Bigl(\frac{1+\sqrt{q}}{2}\Bigr)^2.
		\]
	\end{proposition}

	\begin{proposition}[Type (B4)] \label{pr:B4coprime}
		Let $K=\BQ{p}{q}$ where $p \equiv q \equiv 1 \pmod{4}$ are coprime square-free integers such that $p,q\geq17$. Then $\P(\O_K)\geq 6$.
		
		In particular, the following element has length $6$:
		\[
		\alpha = 7+\Bigl(\frac{1+\sqrt{p}}{2}\Bigr)^2+\Bigl(\frac{1+\sqrt{q}}{2}\Bigr)^2.
		\]
	\end{proposition}
	
	Note that in the last proposition, the field is in fact always of type (B4a) since $\gcd(p,q) \equiv 1 \pmod4$. Let us also remark that the omitted case $p \equiv q \equiv 3$ (see Proposition \ref{pr:coprime}) seems to be somewhat harder, but we believe that $\P(\O_K)\geq 6$ holds even in that situation.

	\section{Pythagoras number is at least 5 (up to seven exceptions)} \label{se:P>=5}
	
	In this section we prove the first half of Theorem \ref{th:main2parts}, i.e.\ that $\P(\O_K) \geq 5$ unless $K = \BQ23$, $\BQ25$, $\BQ35$ (where we expect $\P(\O_K)=3$) or $K = \BQ37$, $\BQ27$, $\BQ56$, $\BQ57$  (where we expect $\P(\O_K)=4$).
	
	The structure of the proof is as follows: The fields which contain $\sqrt2$, $\sqrt3$ or $\sqrt5$ are handled in Subsection \ref{ss:235}, with the most difficult case postponed until Subsection \ref{ss:strongAsymptotics}; the fields containing $\sqrt6$ or $\sqrt7$ are discussed in Subsection \ref{ss:67}. The fields containing $\sqrt{13}$ are postponed until Proposition \ref{pr:sqrt13P>=6}, where we prove that the Pythagoras number is in fact at least $6$. For the fields which contain none of these six numbers, the easier case of $m \equiv 1 \pmod4$ is solved in Subsection \ref{ss:P>=5mEquiv1}, while Subsection \ref{ss:P>=5mEquiv23} solves a majority of the fields where $m \not\equiv 1 \pmod4$. Among these, the fields with $s=m+4$, $s=m+8$ or $s=m+12$ turn out to be the most problematic; a solution for them is provided in Subsection \ref{ss:m+4,m+8,m+12}. A formal proof of the whole Theorem \ref{th:main2parts} (\ref{it:main(1)}), including references to specific propositions rather than subsections, is situated at the very end of Section \ref{se:P>=5}.
	
	Since the proof is a bit lengthy, we first provide a hint how to easily obtain the weaker estimate $\P(\O_K)\geq 4$: By trace considerations, it is simple to show that if $\ell(7)<4$, then $K$ must contain $\sqrt2$, $\sqrt3$, $\sqrt5$, $\sqrt6$, $\sqrt7$ or $\sqrt{13}$. For fields with $m=6$, the element $10+2\sqrt{6}=1+1+1+(1+\sqrt{6})^2$ is a sum of less than four squares only in $\BQ6{10}$, and here one easily finds another counterexample. For fields with $m=7$ or $m=13$, one obtains similar results for $11+2\sqrt{7}=1+1+1+(1+\sqrt{7})^2$ and $\frac{13}{2}+\frac{1}{2}\sqrt{13}=1+1+1+\big(\frac{1+\sqrt{13}}{2}\big)^2$, respectively; only the field $\BQ{7}{11}$ has to be handled separately. It then remains to find a proof for fields containing $\sqrt{2}$, $\sqrt{3}$ or $\sqrt{5}$. This is easily done by the same method which we use in Subsection \ref{ss:67} for proving $\P(\O_K)\geq 5$ for $m=6, 7$; a more general explanation of this strategy is provided later in Subsection \ref{ss:P>=6idea}. 
	
	The rest of the section is spent solely by proving the result advertised in its title. Remember that we keep the notation $1<m<s<t$ for the three square-free integers whose square roots belong to $K$.

	\subsection{The cases where \texorpdfstring{$m \equiv 1 \pmod4$}{m is 1 mod 4}} \label{ss:P>=5mEquiv1}
	The following proposition shows that for $m \equiv 1 \pmod4$, the simplest element of length $5$ in $\Z\bigl[\frac{1+\sqrt{m}}{2}\bigr]$ has length $5$ also in $\O_K$:
	
	\begin{proposition} \label{pr:mIs1}
		Let $K = \BQ{m}{s}$ be a totally real biquadratic field with $m \equiv 1 \pmod{4}$ where $m \neq 5,13$. Then $\P(\O_K)\geq 5$.
		
		In particular, $\ell(\alpha_0)=5$ holds for
		\[
		\alpha_0 = 7 + \Bigl(\frac{1+\sqrt{m}}{2}\Bigr)^2 = 7 + \frac{1+m}{4} + \frac{\sqrt{m}}{2}.
		\]
	\end{proposition}
	\begin{proof}
		Peters \cite{Pe} proved that $\alpha_0$ cannot be expressed as a sum of four squares in $\O_{\Q(\!\sqrt{m})}$, see Theorem \ref{th:quadratic}. Therefore, if $\alpha_0 = \sum_{i=1}^{4}x_i^2$ in $\O_K$, we can assume $x_1 \notin \Q(\!\sqrt{m})$.
		
		First suppose that $x_i \in \Q(\!\sqrt{m})$ for $i\neq 1$. Then $x_1^2 \in \Q(\!\sqrt{m})$, hence by Lemma \ref{le:KTZ} $x_1$ has the form $k\sqrt{s}+l\sqrt{t}$ for $k,l\in\Q$. However, these elements have a big trace: To minimise the trace of $x_1^2$, one has to choose either $x_1 = \sqrt{s}$ or $x_1=\frac{\sqrt{s}\pm\sqrt{t}}{2}$. Therefore, by comparing traces of $\alpha_0$ and $x_1^2$, we obtain one of the inequalities
		\[
		7 + \frac{1+m}{4} \geq s \quad \text{or} \quad 7 + \frac{1+m}{4} \geq \frac{s+t}{4}.
		\]
		The first inequality is nonsense since $s\geq m+1$. As for the second, it clearly implies $t \leq 28$ which holds only for finitely many biquadratic fields; and none of them satisfies $m \equiv 1 \pmod{4}$.
		
		Therefore we can assume $x_1,x_2\notin \Q(\!\sqrt{m})$. We distinguish two cases. In the first, both $x_1$ and $x_2$ are of the form $\frac{a+b\sqrt{m}+c\sqrt{s}+d\sqrt{t}}{2}$ for $a,b,c,d \in \Z$, i.e.\ \enquote{no quarter-integers are involved}. Then $\atr(x_1^2+x_2^2) \geq 2\atr\bigl(\frac{1+\sqrt{s}}{2}\bigr)^2$, yielding
		\[
		7 + \frac{1+m}{4} \geq 2\frac{1+s}{4},
		\]
		i.e.\ $27 + m \geq 2s$. This, together with our conditions on $m$, holds only for the five fields $\BQ{17}{19}$, $\BQ{17}{21}$, $\BQ{17}{22}$, $\BQ{21}{22}$, $\BQ{21}{23}$; these fields were handled by our computer program, see Lemma \ref{le:comp} (\ref{it:mIs1}).
		
		In the other case, one of $x_i$ is of the form $\frac{a+b\sqrt{m}+c\sqrt{s}+d\sqrt{t}}{4}$ with $a,b,c,d$ odd. Then, by Lemma \ref{le:quarterSquares}, $x_i^2$ is of the same form; hence, to get $\alpha_0$, at least one other summand must take the same form. Comparing traces, this yields
		\[
		7 + \frac{1+m}{4} \geq 2 \frac{1+m+s+t}{16},
		\]
		i.e.\ $57+m\geq s+t$, implying $56\geq t$. This is impossible, since quarter-integers can only occur for fields of type (B4), and the smallest value of $t$ in such a field is $65$.
	\end{proof}

	\subsection{The cases where \texorpdfstring{$m \not\equiv 1 \pmod4$}{m is not 1 mod 4}} \label{ss:P>=5mEquiv23}
	
	If we consider the fields where $m \not\equiv 1 \pmod4$, it is natural to ask whether $\alpha_0 = 7 + (1+\sqrt{m})^2$ satisfies $\ell(\alpha_0)=5$. Usually, the answer is positive. However, in contrast with the analogous question for $m\equiv 1$, there are infinitely many fields where this is not the case:
	
	\begin{observation} \label{ob:m+4etcIsAProblem}
		In a field $\BQ{m}{s}$ where $s=m+4$, $s=m+8$ or $s=m+12$, four squares suffice to represent $\alpha_0 = 7 + (1+\sqrt{m})^2 = (8+m) + 2\sqrt{m}$. More specifically, denote $y_1 = 1 + \frac{\sqrt{m}+\sqrt{s}}{2}$ and $y_2 = 1 + \frac{\sqrt{m}-\sqrt{s}}{2}$. Then
		\[
		\alpha_0 = \begin{cases}
			y_1^2+y_2^2+2^2 & \text{if $s = m + 4$},\\
			y_1^2+y_2^2+1^2+1^2 & \text{if $s = m + 8$},\\
			y_1^2+y_2^2 & \text{if $s = m + 12$},
		\end{cases}
		\]
		as one easily sees from the general identity $y_1^2+y_2^2 = 2 + \frac{m+s}{2} + 2\sqrt{m}$.
	\end{observation}

	The three families of fields which were shown to be problematic in Observation \ref{ob:m+4etcIsAProblem} will be handled in the next subsection. The aim of this subsection is to prove that when one excludes these three families (and of course supposes $m \neq 2,3,6,7$), then five squares are indeed needed for $\alpha_0$ in all but three fields.
	
	\begin{proposition} \label{pr:mNot1Pis5}
		Let $K = \BQ{m}{s}$ be a totally real biquadratic field with $m \not\equiv 1 \pmod{4}$ such that $\P(\O_F)=5$ for $F=\Q(\!\sqrt{m})$. (Namely $m \neq 2,3,6,7$.) Assume further that $s \neq m+4,m+8,m+12$. Then $\P(\O_K)\geq 5$.
		
		In particular, the length of
		\[
		\alpha_0 = 7 + (1+\sqrt{m})^2 = (8+m) + 2\sqrt{m}
		\]
		is $5$ unless $K=\BQ{10}{13}$, $\BQ{11}{13}$ or $\BQ{30}{35}$. In the first two of these three fields, $\ell\Bigl(3+(\frac{1+\sqrt{13}}{2})^2+(1+\frac{1+\sqrt{13}}{2})^2\Bigr)=5$; in the last, $\ell\Bigl(7 + (1+\sqrt{42})^2\Bigr)=5$.
	\end{proposition}
	\begin{proof}
		We start by defining three sets of exceptional fields for which the following proof does not work and they must be handled separately. For convenience, denote by $\K$ the set of all fields for which we aim to prove the theorem, i.e.\ fields where $m \not\equiv 1 \pmod4$, $m \neq 2,3,6,7$ and $s \neq m+4,m+8,m+12$. Now, the problematic fields will be defined by the following inequalities:
		\begin{align*}
			\K_1 &= \{K \in \K : 31+3m \geq 4s\},\\
			\K_2 &= \{K \in \K : 31+2m \geq s+t\},\\
			\K_3 &= \{K \in \K : 31+m \geq 2s\}.
		\end{align*}
		Clearly $\K_1 \subset \K_3$. It is also not very difficult to see that all these sets are finite; this is done in Lemma \ref{le:listsOfExceptions}. Once the lists are proven to be finite, a computer program can be applied to solve them; this was done in Lemma \ref{le:comp} (\ref{it:mNot1Gen}) and (\ref{it:threeexceptions}). (Note that the three specific fields listed in the statement of Proposition \ref{pr:mNot1Pis5} belong to $\K_2$ or $\K_3$.) Having gotten this out of the way, let us now proceed with the general proof:
		
		There are several possibilities regarding the integral basis; however, we shall perform the proof for all of them simultaneously. We assume $\alpha_0 = \sum x_i^2$, where
		\[
		x_i = \frac{a_i + b_i\sqrt{m} + c_i\sqrt{s} + d_i\sqrt{t}}{2}, \qquad \text{with $a_i+b_i+c_i+d_i \equiv 0 \pmod2$}.
		\]
		Depending on the integral basis, further conditions are imposed on the coefficients, but we shall repeat those directly when using them. As usual, the core of the proof is comparing of traces; however, in this case, it will not give us enough information. Therefore we start with another observation, which will then make the comparing of traces more efficient.
		
		Since $\alpha_0$ is not a sum of four squares in $\Z[\sqrt{m}]$, we can assume that at least one of $x_i$ is not in $\Z[\sqrt{m}]$. Now we proceed to prove a sequence of claims which further restrict the possible values of all the coefficients $b_i,c_i,d_i$:
		
		\begin{enumerate}[label=(\roman*)]
			\item At least one product $a_kb_k \neq 0$. In particular, $\sum a_i^2 + m\sum b_i^2 \geq 1 + m$.
			\item With the exception of fields $K \in \K_1$, we have $\sum c_i^2 + \sum d_i^2 \leq 3$. In particular, $c_i$ and $d_i$ only take the values $0, \pm1$.
			\item With the same fields $K \in \K_1$ excluded, there are at least two $x_i \notin \Z[\sqrt{m}]$. This also means $\sum c_i^2 + \sum d_i^2 \geq 2$.
			\item If $K \notin \K_1 \cup \K_2$, the following holds: Either all $d_i = 0$, or $\sum b_i^2 + \sum c_i^2 + \sum d_i^2 \leq 3$.
			\item If $K \notin \K_1 \cup \K_3$, we have $\sum b_i^2 + \sum c_i^2 + \sum d_i^2 \leq 4$. In particular, $\sum b_i^2 \leq 2$, so $b_i$ can only take values $0, \pm 1$.
		\end{enumerate}
		
		The proofs follow:
		
		\textbf{(i)} Comparing the coefficients in front of $\sqrt{m}$ yields the equality $2 = 2\bigl(\sum \frac{a_ib_i}{4} + m_0\sum\frac{c_id_i}{4}\bigr)$. Assuming that $a_ib_i=0$ for all indices, we get $4 = m_0 \sum c_id_i$, implying $m_0 \mid 4$. Since $m_0$ is square-free and at least $3$, this is impossible.
		
		\textbf{(ii)} By comparing traces, exploiting the previous claim, we get
		\[
		32 + 4m = 4\atr\alpha = \sum a_i^2 + m\sum b_i^2 + s\sum c_i^2 + t\sum d_i^2 \geq 1 + m + s\Bigl(\sum c_i^2 + \sum d_i^2\Bigr).
		\]
		For the sake of contradiction, assume $\sum c_i^2 + \sum d_i^2 \geq 4$. This yields $31 + 3m \geq 4s$, which is the definition of $K \in \K_1$.
		
		\textbf{(iii)} We already know that at least one $x_i \notin \Z[\sqrt{m}]$. Suppose it were the only one -- say, $x_1$. Then $x_1^2 \in \Z[\sqrt{m}]$, which, by Lemma \ref{le:KTZ}, means $x_1 = \frac{c_1\sqrt{s}+d_1\sqrt{t}}{2}$. \notabene{One can note that since $m\not\equiv 1$, this is an algebraic integer if and only if the field is of type (B1) with $s\equiv t \equiv m_0 \equiv 2 \pmod4$. But the given argument makes no use of this.}
		
		Combining the previous claim with $c_1\equiv d_1 \pmod2$, we get $|c_1|=|d_1|=1$. Thus there are only two possibilities: $x_1^2 = \frac{s+t}{4} \pm \frac{m_0}{2}\sqrt{m}$. Although for $m_0 \equiv 2 \pmod4$ this indeed belongs to $\Z[\sqrt{m}]$, it can never belong to $\Z[2\sqrt{m}]$. Thus $\alpha_0 - x_1^2$ cannot be a sum of squares in $\Z[\sqrt{m}]$, a contradiction.
		
		\textbf{(iv)} To arrive at a contradiction, assume that $\sum b_i^2 + \sum c_i^2 + \sum d_i^2 \geq 4$ where at least one $d_i$ is nonzero. Since $K \notin \K_1$, by the previous claim we also know $\sum c_i^2 + \sum d_i^2 \geq 2$. Putting these pieces of information together, we have $m\sum b_i^2 + s\sum c_i^2 + t \sum d_i^2 \geq 2m + s + t$. Thus
		\[
		32 + 4m \geq 1 + m\sum b_i^2 + s\sum c_i^2 + t \sum d_i^2 \geq 1 + 2m + s + t,
		\]
		i.e.\ $31 + 2m \geq s + t$. This is the definition of $K \in \K_2$.
		
		\textbf{(v)} For the sake of contradiction, assume $\sum b_i^2 + \sum c_i^2 + \sum d_i^2 \geq 5$. By (iii) we also know $\sum c_i^2 + \sum d_i^2 \geq 2$. Putting these pieces of information together, we have $m\sum b_i^2 + s\sum c_i^2 + t \sum d_i^2 \geq 3m + 2s$. Thus
		\[
		32 + 4m \geq 1 + m\sum b_i^2 + s\sum c_i^2 + t \sum d_i^2 \geq 1 + 3m + 2s,
		\]
		i.e.\ $31 + m \geq 2s$. This is the definition of $K \in \K_3$. Now $\sum b_i^2 \leq 2$ follows using claim (iii).
		
		Now, assuming $K\notin \K_1 \cup \K_2 \cup \K_3$, the five proven claims provide very strong restrictions for the possible values of $|b_i|,|c_i|,|d_i|$. One sees that all of them are only zeros or ones. To simplify the following discussion, denote by $B,C,D$ the number of nonzero $b_i,c_i,d_i$. If $D \neq 0$, then by claim (iv) $B+C+D \leq 3$ and by claims (i) and (iii) $B\geq 1$ and $C+D \geq 2$. This gives rise to two possibilities: $B=C=D=1$ or $B=1, C=0, D=2$. On the other hand, if $D = 0$, then claims (i) and (v) give $1 \leq B \leq 2$ and claims (ii) and (iii) yield $2 \leq C \leq 3$. These are four more possibilities.
		
		Of these latter four possibilities, three can be excluded based on examining the integral bases: Since $D=0$, all $d_i$ are zero, hence even. Since $B,C \neq 0$, there are odd values of $b_i$ and of $c_i$. Unless $B=C=2$, the values of $B$ and $C$ are different, so there is some index $i$ where exactly one of $b_i$ and $c_i$ is odd; therefore $a_i$ is odd as well. Hence $a_i,b_i,c_i$ all attain odd values in our summands, while $d_i$ does not. However, this requires at least two of the elements $\frac{1+\sqrt{m}}{2}$, $\frac{1+\sqrt{s}}{2}$, $\frac{\sqrt{m}+\sqrt{s}}{2}$ to be algebraic integers, and this is impossible (it holds only for fields of type (B4), but there $m \equiv 1 \pmod4$).
		
		So of the cases where $D=0$, only $B=C=2$ remains to be solved. Here all $d_i$ are even, while $b_i$ and $c_i$ attain also odd values. In order not to get a contradiction as in the last paragraph, they must attain these values simultaneously, so the integral basis must contain $\frac{\sqrt{m}+\sqrt{s}}{2}$. This happens if and only if $m \equiv s \pmod4$. Now we compare traces one last time, using also the existence of $a_k\neq0$:
		\[
		32 + 4m = \sum a_i^2 + 2m + 2s \geq 1 + 2m + 2s,
		\]
		which can be rewritten as $s - m \leq \frac{31}{2} < 16$. The condition $m \equiv s \pmod4$ now yields $s = m+4, m+8$ or $m+12$, in which cases nontrivial decompositions indeed exist.
		
		The last two cases to be handled are $B=C=D=1$ and $B=1, C=0, D=2$. Let us tackle the latter. There are two odd $d_i$, while only one $b_i$ is odd and all $c_i$ are even. To satisfy the parity condition $a_i+b_i+c_i+d_i \equiv 0 \pmod2$ for both indices where $d_i \neq 0$, at least one $a_i$ must also be odd. This leads to a contradiction similar to one we already saw: For $m\not\equiv 1$, at most one of the elements $\frac{1+\sqrt{m}}{2}$, $\frac{1+\sqrt{t}}{2}$, $\frac{\sqrt{m}+\sqrt{t}}{2}$ can be an algebraic integer.
		
		The last remaining case is $B=C=D=1$. (Here the parity condition is easily satisfied: Either one or three of $a_i$ are odd.) By claim (iii) there are two different summands $x_i\notin \Z[\sqrt{m}]$, so we can assume $c_1=1$, $d_2=1$ (and all the other $c_i, d_i = 0$). There is exactly one nonzero $b_i$, let us denote it by $b_k = \pm 1$. First we realise that $k=1$ or $k=2$: If not, then to satisfy the parity condition, all of $a_1,a_2,a_k$ are odd and all the elements $\frac{1+\sqrt{m}}{2}$, $\frac{1+\sqrt{s}}{2}$, $\frac{1+\sqrt{t}}{2}$ are integral; this is nonsense. So indeed $k=1$ or $k=2$. To conclude, let us consider three of the four equations obtained by comparing $\alpha_0$ with $\sum x_i^2$ coefficient by coefficient:
		\begin{align*}
			4 &= \sum a_ib_i + m_0\sum c_id_i = \pm a_k\\
			0 &= \sum a_ic_i + s_0\sum b_id_i = a_1 + s_0 b_2\\
			0 &= \sum a_id_i + t_0\sum b_ic_i = a_2 + t_0 b_1.
		\end{align*}
		If $k=2$, then $b_1=0$, so the last equation gives $a_k=a_2=0$, which contradicts the first equation. Similarly, if $k=1$, then the second equation contradicts the first.
		
		Thus the proof is concluded (except for the cases $\K_1$, $\K_2$, $\K_3$, which are handled below).
	\end{proof}
	
	Now we handle the exceptional cases for which the inequalities in the previous proof were not good enough, thus completing the proof.
	
	\begin{lemma}\label{le:listsOfExceptions}
		Denote by $\K$ the set of all biquadratic fields where $m \not\equiv 1 \pmod4$, $m \neq 2,3,6,7$ and $s \neq m+4,m+8,m+12$. Then the following sets are finite:
		\begin{align*}
			\K_1 &= \{K \in \K : 31+3m \geq 4s\},\\
			\K_2 &= \{K \in \K : 31+2m \geq s+t\},\\
			\K_3 &= \{K \in \K : 31+m \geq 2s\}
		\end{align*}
		and for each $K \in \K_1 \cup \K_2 \cup \K_3$, the claim of the previous proposition holds: Namely, $\alpha_0 = 7 + (1+\sqrt{m})^2$ has length $5$ in $\O_K$ unless $K$ is one of the three exceptional fields; these three fields contain another element of length $5$, defined in the previous proposition.
	\end{lemma}
	\begin{proof}
		Clearly $\K_1 \subset \K_3$, so it suffices to consider the larger sets $\K_2$ and $\K_3$. We already explained in Subsection \ref{ss:algorithms} that in a \emph{fixed} order $\O_K$, computing the length of a given element is just a computational task. Thus to prove our result, it suffices to provide a full list of fields in $\K_2$ and $\K_3$; the rest of the proof can be handled by a simple computer program (or a very patient student). The results of this computation are already prepared in Lemma \ref{le:comp} (\ref{it:mNot1Gen}) and (\ref{it:threeexceptions}).
		
		We start with $\K_3$. The defining inequality implies $30 \geq s$, so we have $10 \leq m \leq 29$. It suffices to go through all square-free $m \not\equiv 1 \pmod4$ in this interval and for each of them to list all square-free $s$ satisfying $m < s  \leq \frac{31+m}{2}$. We also omit cases where $s=m+4,m+8,m+12$. The full list of the eighteen pairs $(m,s)$ such that $\BQ{m}{s} \in \K_3$ follows:
		\begin{align*}
			&(10,11), (10,13), (10,17), (10,19); (11,13), (11,14), (11,17), (11,21); (14,15), (14,17),\\
			& (14,19); (15,17), (15,21), (15,22); (19,21), (19,22); (22,23); (23,26).
		\end{align*}
		We omitted the fields $(10,15)$ and $(14,21)$, since we require $t$ to be the largest of the three number $m,s,t$, while these fields contain $\sqrt6$. Therefore, they do not belong to $\K$ and are handled in Subsection \ref{ss:67}.
		
		The finiteness of $\K_2$ is less obvious. However, using the notation with $t_0<s_0<m_0$, one can write
		\[
		31 \geq s + t - 2m = m_0t_0 + m_0s_0 - 2s_0t_0 \geq m_0t_0 + m_0(t_0+1) - 2(m_0-1)t_0 = m_0 + 2t_0,
		\]
		so the values of $m_0$ and hence also $s_0$ and $t_0$ are bounded. With a little care, it is quite fast and simple to find a complete list by hand. \notabene{The most convenient way of listing all triples $(t_0,s_0,m_0)$ which correspond to a field in $\K_2$ is probably to observe that the above inequality implies $t_0 \leq \frac{29}{3} < 10$, so since $t_0$ is square-free, one has $t_0 \in \{1,2,3,5,6,7\}$. Going over these possibilities one by one, one has to look for solutions $m,s$ (which are coprime, square-free, and multiples of $t_0$) of inequalities obtained by plugging the given value of $t_0$ into $31 \geq s + m\bigl(\frac{s}{t_0^2} - 2\bigr)$. For example, if $t_0=2$, then the values $m=10$ and $s=14$ are the smallest possible, and they satisfy the inequality -- so they would belong to $\K_2$, were it not for the fact that $s=m+4$ is forbidden. Otherwise $m \geq 14$ and $s \geq 22$ (other even numbers are not square-free), but $s + m\bigl(\frac{s}{4} - 2\bigr) \geq 22 + 14\bigl(\frac{22}{4}-2\bigr) = 71 > 2$, so there are no other solutions with $t_0=2$.}
		
		It turns out that, due to the restriction $m \geq 10$, the set $\K_2$ contains only two fields:
		\[
		\K_2 = \bigl\{\BQ{15}{21}, \BQ{30}{35} \bigr\}.
		\]
		Once we have made the lists explicit, we can apply our computer program. This was successfully done, see Lemma \ref{le:comp} (\ref{it:mNot1Gen}) and (\ref{it:threeexceptions}). In the three fields given by $(10,13)$, $(11,13)$ and $(30,35)$, the original $\alpha_0$ turns out to be a sum of four or even three squares, hence we had to find another element of length $5$.
	\end{proof}

	\subsection{The exceptional cases \texorpdfstring{$s=m+4, m+8, m+12$}{s=m+4, m+8, m+12}} \label{ss:m+4,m+8,m+12}
	
	In this subsection we want to prove $\P(\O_K)\geq 5$ for the three families of fields given by $s=m+4$, $s=m+8$ and $s=m+12$ (keeping the conditions $m \not\equiv 1 \pmod4$, $m \neq 2,3,6,7$). For these fields, the proof from the previous subsection cannot be used, since $\alpha_0 = 7 + (1+\sqrt{m})^2$ can be written as a sum of four or less squares, see Observation \ref{ob:m+4etcIsAProblem}. To find a solution of this problem, let us first examine why the main proof failed in these cases.
	
	The problem lies in the identity
	\begin{align*}
		\alpha_0 - \Bigl(1 + \frac{\sqrt{m}+\sqrt{s}}{2}\Bigr)^2-\Bigl(1 + \frac{\sqrt{m}-\sqrt{s}}{2}\Bigr)^2 &= (7 + 1 + m + 2\sqrt{m}) - \bigl(2 + \frac{m+s}{2} + 2\sqrt{m}\bigr)\\
		&= 7 - 1 - \frac{s-m}{2},
	\end{align*}
	where $7 - 1 - \frac{s-m}{2}$ is either $7 - 3 = 4$, $7 - 5 = 2$ or $7 - 7 = 0$, depending on the concrete family. All the three values, $4$, $2$ and $0$, are sums of at most two squares in $\Z$.
	
	In this, we were actually rather unlucky: Most positive numbers require three or even four squares! So one solution which suggests itself is to replace $7$ by another number $n$ which shares the crucial property of $7$ (namely, it is not a sum of three squares in $\Z$), but also solves the previous problem since $n - 3$, $n - 5$ or $n - 7$ (again depending on the family we are examining) is not a sum of two squares.
	
	The next number after $7$ which is not a sum of three squares is $15$; and while $15-5 = 1^2 + 3^2$ and $15 - 7 = 2^2 + 2^2$ are sums of two squares, $15 - 3$ is not. Therefore the first family, $s = m + 4$, could be solved using the element $15 + (1 + \sqrt{m})^2$. The next number which is not a sum of three squares, $23$, does not help, but for $28$ one sees that both $28 - 5 = 23$ and $28 - 7 = 21$ can not be represented by two squares. Thus, in the following proposition, we solve the remaining two families thanks to the element $28 + (1+\sqrt{m})^2$.
	
	\begin{proposition} \label{pr:m+4etc}
		Let $K = \BQ{m}{s}$ where $m \not\equiv 1 \pmod4$, $m \neq 2,3,6,7$, and $s=m+4$, $s=m+8$ or $s=m+12$. Then $\P(\O_K)\geq 5$.
		
		In particular,
		\[
		\alpha = \begin{cases}
			15 + (1+\sqrt{m})^2 & \text{if $s = m + 4$},\\
			28 + (1+\sqrt{m})^2 & \text{if $s = m + 8$ or $m + 12$}\\
		\end{cases}
		\]
		has length $5$ unless $K$ is one of the following $16$ fields: $m = 10, 11, 15$ and $s= m+4$; $m = 14, 22, 26, 11, 15, 23$ and $s= m+8$; $m = 10, 14, 22, 26, 11, 19, 23$ and $s= m+12$. In these exceptional fields, $7 + \bigl(\frac{\sqrt{m}+\sqrt{s}}{2}\bigr)^2$ has length $5$.\notabene{It seems that this element actually solves emph{all} such fields, not only the $16$ exceptions; it is probably just better than $7 + (1+\sqrt{m})^2$, as soon as $s-m$ is small. Why? One must ask which of $\atr (1+\sqrt{m})^2 = 1 + m$ and $\atr \Bigl(\frac{\sqrt{m}+\sqrt{s}}{2}\Bigr)^2 = \frac{m+s}{4}$ is smaller. That depends on an inequality like $3m \leq s$. Equivalently, $3s_0 \leq m_0$. $\ldots$ Yes, similar thoughts are behind the distinction of $12$ families in the last section.}
	\end{proposition}
	
	\begin{proof}
		The proof goes along the same lines as that of Proposition \ref{pr:mNot1Pis5}, only at the end one must prove that it is impossible to have two summands of the form $x_i = \frac{a_1+b_1\sqrt{m}+\sqrt{s}\pm \sqrt{t}}{2}$. We just go through the claims which have to be proven on the way, to see which exceptional fields must be handled separately. (See Lemma \ref{le:comp}, not only for the 16 exceptional fields listed in (\ref{it:mNot1SpecBad}) but also the fields included in (\ref{it:mNot1SpecGood15}) and (\ref{it:mNot1SpecGood28}), for which the statement holds but the following proof breaks down.)
		
		As always, we suppose $\alpha = \sum x_i^2$ with $x_i = \frac{a_i+b_i\sqrt{m}+c_i\sqrt{s}+d_i\sqrt{t}}{2}$; since $m\equiv s \not\equiv 1 \pmod{4}$, the conditions are $a_i\equiv d_i$ and $b_i \equiv c_i \pmod2$ (and in the cases where $t \equiv 3$, $a_i$ and $d_i$ must both be even). It is important to note $t = \frac{ms}{t_0^2}$, where $t_0$, as the greatest common divisor of $m$ and $s$, divides $4$, $8$ or $12$, respectively. Together with $t_0$ being square-free, one sees that if $m$ and $s$ are even, then $t_0 = 2$ or $6$, the latter case occurring only if $s = m+12$ and $3 \mid m$. Similarly, if $m$ and $s$ are odd (thus $3$ modulo $4$), then $t_0 =1$ or $3$, the latter case again occurring only if $s=m+12$ and $3 \mid m$. From this, it is easy to see whether $t \equiv 1$ or $t \equiv 3 \pmod4$, determining the integral basis.
		
		\textbf{(A) $s=m+4$:} The three fields with $m \leq 15$ must have been handled separately (Lemma \ref{le:comp} (\ref{it:mNot1SpecBad})) since $\alpha$ is a sum of less than five squares already in $\Z[\sqrt{m}]$. From now on suppose $m>15$. Then at least one $x_i$ is not in $\Z[\sqrt{m}]$. Also, by comparing coefficients in front of $\sqrt{m}$, we see that at least one $a_ib_i \neq 0$. Suppose now that there is a nonzero $d_i$. For odd $m$ we have $\frac{d_i^2 t}{4} = \frac{d_i^2}{4}m(m+4) \geq \frac14 m(22+4) = \frac{13}{2}m$, while for even $m$ we use the fact that $t = \frac{m(m+4)}{4} \equiv 3 \pmod4$ to deduce $\frac{d_i^2 t}{4} \geq \frac{2^2}{4}\frac{m(m+4)}{4} \geq m \frac{19+4}{4}$. Either way, $\frac{d_i^2 t}{4}$ is larger than $\atr \alpha = 16 + m$, which is a contradiction. Hence, all $d_i$ are zero.
		
		Now this implies that all $a_i$ must be even, and one can forget about the integral basis since $t$ plays no more role. If there were only one summand with $c_i \neq 0$, its square would have to belong to $\Z[\sqrt{m}]$, so by Lemma \ref{le:KTZ}, it would be $\frac{c_i}{2}\sqrt{s}$ with $c_i$ even; we also know that at least one $b_j$ is nonzero, and since $c_j = 0$, this $b_j$ is even. Putting this together, if there is only one summand with $c_i \neq 0$, comparing traces yields $16 + m \geq \frac{b_j^2}{4}m + \frac{c_i^2}{4}s \geq m + s = 2m + 4$, which is impossible. 
		
		Therefore there are at least two summands with $c_i \neq 0$. This also means $\sum c_i^2 \geq 2$, and due to the congruence conditions $\sum b_i^2 + \sum c_i^2 \geq 4$. The next step is to show that $\sum b_i^2 + \sum c_i^2 > 4$ is impossible; and indeed, comparing traces and exploiting $\sum b_i^2 + \sum c_i^2 \geq 6$ and $\sum c_i^2 \geq 2$ leads to a contradiction for $m > 26$ (leaving a few cases to be checked separately, see Lemma \ref{le:comp} (\ref{it:mNot1SpecGood15})).
		
		Now, finally, we know $\sum b_i^2 + \sum c_i^2 = 4$, which together with the existence of nonzero $b_i$ and nonzero $c_j$ implies that one can assume $c_1=c_2=1$, $|b_1|=|b_2|=1$, and all other $b_i$, $c_i$, $d_i$ are zeros. As a result of this, the equations obtained by comparing coefficients of $\alpha$ and $\sum x_i^2$ take the simple form (with the notation $A_i=\frac{a_i}{2} \in \Z$):
		\begin{align*}
			16 + m &= \sum A_i^2 + m\Bigl(\frac14 + \frac14\Bigr) + s\Bigl(\frac14 + \frac14\Bigr) = \sum A_i^2 + m + 2,\\
			2 &= 2\sum \frac{A_ib_i}{2} = A_1b_1 + A_2b_2,\\
			0 &= \sum A_ic_i = A_1 + A_2,\\
			0 &= \sum b_ic_i = b_1 + b_2.\\
		\end{align*}
		The last three equations together with $|b_1|=1$ show that (after possibly switching $x_1$ and $x_2$) $A_1=b_1=1$, $A_2=b_2=-1$, which turns the first equation into
		\[
		12 = \sum_{i\geq 3} A_i^2.
		\]
		From this, we see that it is necessary (and sufficient) that there are at least three more summands $A_3$, $A_4$, $A_5$. We have proven that $\alpha$ is not a sum of less than five squares.
		
		\textbf{(B) $s = m+8$:} For the cases where $m \leq 28$, another choice of $\alpha$ had to be made (Lemma \ref{le:comp} (\ref{it:mNot1SpecBad})), since this $\alpha$ is a sum of less than five squares in $\Z[\sqrt{m}]$. Otherwise the proof is just the same as for $s=m+4$. Proving that all $d_i$ must be zero is slightly complicated by the fact that both for odd and even $m$, we get $t \equiv 1$; but since in the even case we have $m \geq 30$, one still gets a contradiction if $d_{i_0}\neq 0$:
		\[
		29 + m \geq \frac{{d_{i_0}}^2}{4}t \geq \frac{1}{4}\frac{m(m+8)}{4} \geq m \frac{38}{16} \geq m + \frac{11}{8}m > m + m \geq m + 30.
		\]
		After that, one goes through the same steps as for $s=m+4$, so we omit most of them. The inequality $\sum b_i^2 + \sum c_i^2 > 4$ gives a contradiction for $m > 48$, so the cases with $m \leq 48$ must be handled separately (Lemma \ref{le:comp} (\ref{it:mNot1SpecGood28})). For $m > 48$, we eventually come to the conclusion that there must be the summands $x_1^2 = \Bigl(1 + \frac{\sqrt{m}+\sqrt{s}}{2}\Bigr)^2$ and $x_2^2 = \Bigl(-1 + \frac{-\sqrt{m}+\sqrt{s}}{2}\Bigr)^2$, and the other summands are squares of rational integers $A_i$ satisfying $\sum_{i\geq 3} A_i^2 = \alpha - x_1^2 - x_2^2 = 23$. Thus there are at least four such summands -- so any representation of $\alpha$ as a sum of squares which uses $x_1 \notin \Z[\sqrt{m}]$ actually needs at least \emph{six} squares.
		
		\textbf{(C) $s = m+12$:} Here the proof is absolutely the same as in the previous case (even including the fact that for $m \leq 28$ there is a representation of $\alpha$ as a sum of less than five squares in $\Z[\sqrt{m}]$).
		
		The only step where we have to be careful is showing that all $d_i$ are zero. If $s = m+12$, then $t_0$ can take the values $1,2,3,6$, and especially in the last case, $t = \frac{m(m+12)}{36}$ is not \emph{much larger} then $m$, which is what the estimates make use of. However, if $t_0$ is $2$ or $6$, then $t \equiv 3\pmod4$, so $\frac{d_i^2}{4}t \geq t$; and since for $m > 29$ the first field with $s=m+12$ and $t_0=6$ is actually $m=66$, $s=78$ (due to $54$ not being square-free and the field $\BQ{30}{42}$ satisfying $t=m+12$, since $s=35)$, one gets the following contradiction:
		\[
		29 + m \geq \frac{d_i^2}{4}t \geq t = \frac{m(m+12)}{36} \geq m \frac{78}{36} > 2m,
		\]
		implying $29 > m$. For $t_0 = 1, 2, 3$ one also gets the inequality $\frac{d_i^2}{4}t > 2m$, leading to the same contradiction, so all $d_i$ must indeed be zeros.
		
		To disprove $\sum b_i^2 + \sum c_i^2 > 4$, one must assume $m > 44$ (the remaining cases being handled in Lemma \ref{le:comp} (\ref{it:mNot1SpecGood28})). With this assumption, one again concludes that one of the summands is $x_1^2 = \Bigl(1 + \frac{\sqrt{m}+\sqrt{s}}{2}\Bigr)^2$, another is $x_2^2 = \Bigl(-1 + \frac{-\sqrt{m}+\sqrt{s}}{2}\Bigr)^2$, and the other summands are squares of rational integers $A_i$. They must satisfy $\sum_{i\geq 3} A_i^2 = \alpha - x_1^2 - x_2^2 = 21$, so there must be at least three of them. Any representation of $\alpha$ as a sum of squares therefore needs at least five squares.
	\end{proof}

	\subsection{Fields containing \texorpdfstring{$\sqrt6$ or $\sqrt7$}{root of 6 or 7}} \label{ss:67}
	
	In this section\notabene{In the \LaTeX source code, behind this section we kept an older version containing also $\sqrt{13}$.} we prove that, save a few exceptions, rings of integers containing $\sqrt6$ or $\sqrt7$ have Pythagoras number at least $5$. More specifically, we prove Propositions \ref{pr:sqrt7} and \ref{pr:sqrt6}. In these statements, the meaning of $s$ is fixed by Convention \ref{no:mst}: We have $m=7$ ($m=6$, resp.), $s>m$ square-free, and the condition $t>s$ translates into $7 \nmid s$ ($3 \nmid s$, resp.).
	
	\begin{proposition} \label{pr:sqrt7}
		Let $K=\BQ{7}{s}$. 
		Then $\P(\O_K)\geq 5$.
		
		In particular, if $s\neq 11$, then $\ell(\alpha)=5$ holds for $\alpha = \alpha_0 + w^2$, where $\alpha_0 = 1^2+1^2+1^2+(1+\sqrt7)^2 = 11 + 2\sqrt7$ and
		\[
		w = \begin{cases}
			\frac{1+\sqrt{s}}{2} & \text{if $s \equiv 1 \pmod4$},\\
			1+\sqrt{s} & \text{if $s \equiv 2 \pmod4$},\\
			\frac{\sqrt{7}+\sqrt{s}}{2} & \text{if $s \equiv 3 \pmod4$}.\\
		\end{cases}
		\]
		For $s=11$, the element  $3+\bigl(\frac{\sqrt{7}+\sqrt{11}}{2}\bigr)^2+\bigl(1+\frac{\sqrt{7}+\sqrt{11}}{2}\bigr)^2$ has length $5$.
	\end{proposition}
	\begin{proof}
		Suppose that $\alpha = \sum x_i^2$; since $\alpha \notin \Q(\!\sqrt7)$, we can assume $x_1^2 \notin \Q(\!\sqrt7)$. It is known (and easy to check) that $\ell(\alpha_0)=4$ in $\Z[\sqrt7]$, see Theorem \ref{th:quadratic}. Therefore, it suffices to prove that $x_1 = \pm w$ and $x_i \in \Z[\sqrt7]$ for $i \neq 1$; indeed, in such a case $\alpha_0 = \sum_{i\neq 1} x_i^2$ requires at least four summands.
		
		To achieve this, denote $x_i = \frac{a_i+b_i\sqrt{7}+c_i\sqrt{s}+d_i\sqrt{7s}}{2}$; there are congruence conditions imposed on the coefficients, but these depend on the type of the field and we do not need them just yet. Comparing traces yields
		\begin{equation}\label{eq:ineq7}
			11 + \atr w^2 = \atr \alpha = \sum \frac{a_i^2 + 7b_i^2 + sc_i^2 + 7sd_i^2}{4} \geq \frac{s}{4}\Bigl(\sum c_i^2 + 7 \sum d_i^2\Bigr).
		\end{equation}
		Since $x_1 \notin \Z[\sqrt7]$, the value of $\sum c_i^2 + 7 \sum d_i^2$ is nonzero. To examine this quantity further, we must distinguish the three possible types of fields:
		
		\textbf{(A) $s \equiv 1$:} The field is of type (B2,3) with $q=s$, which means that $a_i\equiv c_i$ and $b_i \equiv d_i \pmod2$. Therefore the smallest nonzero value of $\sum c_i^2 + 7 \sum d_i^2$ is $1$. Suppose first, for the sake of contradiction, that $\sum c_i^2 + 7 \sum d_i^2 > 1$. Then we have at least $\sum c_i^2 + 7 \sum d_i^2 \geq 2$. Plugging this, together with $\atr w^2 = \frac{1+s}{4}$, into the inequality \eqref{eq:ineq7}, yields
		\[
		11 + \frac{1+s}{4} \geq \frac{s}{4}2,
		\]
		i.e.\ $45 \geq s$. This holds only for six fields, handled separately in Lemma \ref{le:comp} (\ref{it:sqrt7}); otherwise we get a contradiction.
		
		Thus we proved $\sum c_i^2 + 7\sum d_i^2 = 1$; together with $x_1 \notin \Q(\!\sqrt7)$ this means $|c_1|=1$ and all other $c_i, d_i = 0$. In particular, $x_i \in \Z[\sqrt7]$ for $i \neq 1$. By choosing the sign, we may assume $x_1 = \frac{a_1 + b_1\sqrt7 + \sqrt{s}}{2}$. This yields
		\[
		\alpha = x_1^2 + \underbrace{\cdots\cdots}_{\in\Q(\!\sqrt7)} = \underbrace{\cdots\cdots}_{\in\Q(\!\sqrt7)} + \frac{a_1}{2}\sqrt{s} + \frac{b_1}{2}\sqrt{7s}.
		\]
		Comparing this with $\alpha = \bigl(11 + \frac{1+s}{4}\bigr) + 2\sqrt7 + \frac{\sqrt{s}}{2}$, we conclude $a_1=1$, $b_1=0$. Thus indeed $x_1 = \frac{1+\sqrt{s}}{2} = w$.
		
		\textbf{(B) $s \equiv 2$:} The proof is the same as in the previous case, only with different congruence conditions: The field is of type (B1) with $q=7$, which means that $a_i, b_i$ are even and $c_i \equiv d_i \pmod2$. Therefore the smallest nonzero value of $\sum c_i^2 + 7 \sum d_i^2$ is $4$. Suppose first, for the sake of contradiction, that $\sum c_i^2 + 7 \sum d_i^2 > 4$. Then we have at least $\sum c_i^2 + 7 \sum d_i^2 \geq 8$, since either some $d_j$ is nonzero, giving $c_j^2+7d_j^2 \geq 1 + 7$, or there are at least two nonzero even $c_i$. Plugging this, together with $\atr w^2 = 1+s$, into the inequality \eqref{eq:ineq7}, yields
		\[
		11 + (1+s) \geq \frac{s}{4}8,
		\]
		i.e.\ $12 \geq s$. This holds only for $s=10$, handled separately in Lemma \ref{le:comp} (\ref{it:sqrt7}); otherwise we get a contradiction.
		
		Thus $\sum c_i^2 + 7\sum d_i^2 = 4$; as before, this means that $x_i \in \Z[\sqrt7]$ for $i\neq 1$ and (after possibly changing the sign) $x_1 = \frac{a_1 + b_1\sqrt7}{2} + \sqrt{s}$. Hence
		\[
		\alpha = x_1^2 + \underbrace{\cdots\cdots}_{\in\Q(\!\sqrt7)} = \underbrace{\cdots\cdots}_{\in\Q(\!\sqrt7)} + a_1\sqrt{s} + b_1\sqrt{7s}.
		\]
		Comparing this with $\alpha = (11 + 1+s) + 2\sqrt7 + 2\sqrt{s}$, we conclude $a_1=2$, $b_1=0$. Thus indeed $x_1 = 1+\sqrt{s} = w$.
		
		\textbf{(C) $s \equiv 3$:} This is analogous to the previous two cases, so we provide only stepping stones. This time the field is of type (B2,3) with $q=7s$, so $a_i \equiv d_i$ and $b_i\equiv c_i \pmod2$; as in the first case, the minimal nonzero value of $\sum c_i^2 + 7\sum d_i^2$ is $1$. The assumption $\sum c_i^2 + 7\sum d_i^2 \geq 2$ leads to $51 \geq s$; this holds only for nine fields, of which eight are handled in Lemma \ref{le:comp} (\ref{it:sqrt7}) and $s=11$ in Lemma \ref{le:comp} (\ref{it:sqrt7spec}). Thus $\sum c_i^2 + 7\sum d_i^2 = 1$, which straightforwardly leads to $x_i \in \Z[\sqrt7]$ for $i \neq 1$ and $\pm x_1 = \frac{a_1 + b_1\sqrt{7} + \sqrt{s}}{2}$; comparing coefficients then yields $\pm x_1 = \frac{\sqrt{7}+\sqrt{s}}{2} = w$ as needed.
	\end{proof}
	
	By straightforwardly applying the same method, we can prove the analogous result for $\sqrt6$. The only difference is that there are more cases to handle due to the fact that $s$ is not necessarily coprime with $6$. However, once the case distinction is made and the integral basis fixed, the proofs are complete analogies to the one above. Note that we had to include $m=6$ explicitly in the statement to exclude the one possibility $t=6$, $s=3$, $m=2$.
	
	\begin{proposition} \label{pr:sqrt6}
		Let $K=\BQ{6}{s}$ with $m=6$. Then $\P(\O_K)\geq 5$.
		
		In particular, if $s\neq 10$, then $\ell(\alpha)=5$ holds for $\alpha = \alpha_0 + w^2$, where $\alpha_0 = 1^2+1^2+1^2+(1+\sqrt6)^2 = 10 + 2\sqrt6$ and
		\[
		w = \begin{cases}
			\frac{1+\sqrt{s}}{2} & \text{if $s \equiv 1 \pmod4$},\\
			\frac{\sqrt{6}+\sqrt{s}}{2} & \text{if $s \equiv 2 \pmod4$},\\
			1+\sqrt{s} & \text{if $s \equiv 3 \pmod4$}.\\
		\end{cases}
		\]
		For $s=10$, the element $3+\bigl(\sqrt{6} -\frac{\sqrt{6}+\sqrt{10}}{2}\bigr)^2+\bigl(2+\frac{\sqrt{6}+\sqrt{10}}{2}\bigr)^2$ has length $5$.
	\end{proposition}
	\begin{proof}
		Direct analogy of the previous proof. If $s \equiv 1,3 \pmod4$, then $\gcd(6,s)=1$, so there is nothing new and one gets the inequalities $s \leq 41$ and $s \leq 22$, respectively, which leaves $5+3$ fields to be checked separately in Lemma \ref{le:comp} (\ref{it:sqrt6}). If $s \equiv 2$, one has $\gcd(6,s)=2$ and $s = 2m_0$, which makes the estimates slightly less efficient; moreover, one must distinguish between fields with $m_0 \equiv 1 \pmod4$, which are of type (B1), and fields with $m_0 \equiv 3 \pmod4$, which are of type (B2,3). The case $s=10$ is dealt with in Lemma \ref{le:comp} (\ref{it:sqrt6spec}), while the $2+7$ fields which are not excluded by the analogously obtainable inequalities $s \leq 2\cdot 23$ in the former and $s \leq 2\cdot 46$ in the latter case, are again solved in Lemma \ref{le:comp} (\ref{it:sqrt6}).
	\end{proof}

	\subsection{Fields containing \texorpdfstring{$\sqrt2$, $\sqrt3$ or $\sqrt5$}{root of 2, 3 or 5}} \label{ss:235}
	
	Finally, we shall handle the fields with $m=2,3,5$. Here we find the result quite surprising; the Pythagoras number of $\O_{\Q(\!\sqrt{m})}$ is only three, but it still turns out that save a few exceptions, the Pythagoras number of $\O_K$ where $K$ contains $\Q(\!\sqrt{m})$ is at least five. This is in contrast with the situation in $\Z$: $\P(\Z)=4$, while in orders of quadratic extensions of $\Z$ the Pythagoras number grows only by (at most) one. Let us also remark that for $m=5$, the Pythagoras number is actually \emph{at most} five, and for $m=2$ we conjecture the same, while for $m=3$ it seems that usually $\P(\O_K)=6$ -- see Theorem \ref{th:mainSqrt5} and Conjecture \ref{co:conjecture}.
	
	The proofs in this subsection are reasonably straightforward (although some tricks like using Lemmata \ref{le:KTZ} and \ref{le:quarterSquares} are required) and not even very lengthy; only guessing the appropriate choice of $\alpha$ would be difficult. We used data from our computer programs (see Subsection \ref{ss:algorithms}) to observe which choice might be the right one. In the case of $K = \BQ5{s}$ with $s \not\equiv 1 \pmod4$, we were unable to find such a suitable $\alpha$;\notabene{If such a choice exists at all, it is \emph{much} less obvious than those which are scattered throughout this subsection.} therefore we had to develop stronger tools -- the proof of this case is given in Subsection \ref{ss:strongAsymptotics}.
	
	\begin{proposition} \label{pr:sqrt5sIs1}
		Let $K=\BQ{5}{s}$ for $s \equiv 1 \pmod{4}$. Then $\P(\O_K)\geq 5$.
		
		In particular, the following element has length $5$:
		\[
		\alpha = 1^2 + 1^2 + \Bigl(\frac{1+\sqrt5}{2}\Bigr)^2 + \Bigl(\frac{1+\sqrt{s}}{2}\Bigr)^2 + \Bigl(\frac{-\sqrt5+\sqrt{s}}{2}\Bigr)^2.
		\]
	\end{proposition}
	\begin{proof}
		Let us suppose that
		\[
		\alpha = \Bigl(5 + \frac{s}{2}\Bigr) + \frac{\sqrt5}{2} + \frac{\sqrt{s}}{2} - \frac{\sqrt{5s}}{2} = \sum x_i^2.
		\]
		First assume that at least one of $x_i$ is of the form $\frac{a_i+b_i\sqrt5+c_i\sqrt{s}+d_i\sqrt{5s}}{4}$ for $a_i,b_i,c_i,d_i$ odd. Then, by Lemma \ref{le:quarterSquares}, $x_i^2$ is of the same form, which implies that at least one other $x_j$ must have the same property as well. Then by comparing traces we obtain
		\[
		5 + \frac{s}{2} \geq \atr(x_i^2 + x_j^2) \geq 2 \cdot \frac{1+5+s+5s}{16} = \frac{3+3s}{4},
		\]
		i.e.\ $17 \geq s$. This is a contradiction unless $s=13$ or $s=17$, which cases are handled in Lemma \ref{le:comp} (\ref{it:sqrt5sIs1}).
		
		So we have shown that all summands are in fact of the form
		\[
		x_i = \frac{a_i+b_i\sqrt5+c_i\sqrt{s}+d_i\sqrt{5s}}{2} \qquad \text{where $a_i+b_i+c_i+d_i$ is even}.
		\]
		Our next aim is to prove that $d_i=0$ for each $i$. Again, this is achieved by comparing traces: If at least one $d_i$ is nonzero, we get
		\[
		5 + \frac{s}{2} \geq 5s \sum \frac14 d_i^2 \geq \frac{5s}{4};
		\]
		this holds only for $s<7$, so in our situation all $d_i$ must indeed be zero.
		
		What can be the value of $\sum c_i^2$? It cannot be zero, since $\alpha \notin \Q(\!\sqrt5)$. The cases $\sum c_i^2 = 1$ and $\sum c_i^2 = 2$ require the most care; now we shall prove that $\sum c_i^2 \geq 3$ is impossible (for $s>10$). Indeed, it would yield
		\[
		5 + \frac{s}{2} \geq s \frac14 \sum c_i^2 \geq \frac{3s}{4}.
		\]
		
		Next we exclude the case $\sum c_i^2 = 1$. Without loss of generality we assume $x_1 = \frac{a_1+b_1\sqrt{5} + \sqrt{s}}{2}$; $a_1 \not\equiv b_1 \pmod2$. All the other summands belong to $\Q(\!\sqrt5)$, so we must have $\alpha - x_1^2 \in \Q(\!\sqrt5)$. After comparing the third and fouth coefficient of
		\[
		x_1^2 = \frac{a_1^2 + 5b_1^2 + s}{4} + \frac{a_1b_1}{2}\sqrt5 + \frac{a_1}2\sqrt{s} + \frac{b_1}{2}\sqrt{5s}
		\]
		with those of $\alpha$, we see $a_1 = 1$, $b_1 = -1$. This is a contradiction, since one of them must be even.
		
		So we have proven that $\sum c_i^2 =2$ is necessary. Without loss of generality we write $x_1 = \frac{a_1+b_1\sqrt{5} + \sqrt{s}}{2}$ and $x_2 = \frac{a_2+b_2\sqrt{5} + \sqrt{s}}{2}$ with $a_1 \not\equiv b_1$ and $a_2\not\equiv b_2 \pmod2$; the remaining $x_i$ are in $\Q(\!\sqrt5)$. Comparing traces, we get the equality
		\[
		5 + \frac{s}{2} = \frac14 \Bigl( \sum a_i^2 + 5 \sum b_i^2 + s \cdot 2 \Bigr),
		\]
		i.e.\ $20 = \sum a_i^2 + 5\sum b_i^2$. This yields a short list of possible combinations of $x_i$, and it is a simple matter to check by hand that the only one which indeed gives $\alpha$ is the one by which $\alpha$ was defined:
		
		
		Since $\alpha - (x_1^2+x_2^2) \in \Q(\!\sqrt5)$, we compare coefficients:
		\[
		x_1^2+x_2^2 = \underbrace{\cdots\cdots}_{\in \Q(\!\sqrt5)} + \frac{a_1+a_2}{2}\sqrt{s} + \frac{b_1+b_2}{2}\sqrt{5s},
		\]
		so $a_1+a_2=1$ and $b_1+b_2=-1$. The latter equality together with $\sum b_i^2 \leq 4$ clearly gives only one solution: $b_1=-1, b_2=0$ (or vice versa). Then $a_1$ is even, and $a_2 = 1 - a_1$ is odd.
		
		To satisfy the inequality $a_1^2 + a_2^2 \leq 20 - 5b_1^2 = 15$, the only options are $a_1 = -2, 0, 2$. Thus there are only three possibilities:
		
		\begin{itemize}
			\item $x_1 = \frac{-2-\sqrt{5}+\sqrt{s}}{2}$, $x_2 = \frac{3 + \sqrt{s}}{2}$,
			\item $x_1 = \frac{-\sqrt{5}+\sqrt{s}}{2}$, $x_2 = \frac{1 + \sqrt{s}}{2}$,
			\item $x_1 = \frac{2-\sqrt{5}+\sqrt{s}}{2}$, $x_2 = \frac{-1 + \sqrt{s}}{2}$.
		\end{itemize}
		In the first case, $\alpha - x_1^2 - x_2^2 = \frac12 - \frac{\sqrt5}2$, which is not totally positive, hence cannot be a sum of squares -- a contradiction. Similarly, in the third case $\alpha - x_1^2 - x_2^2 = \frac52 + \frac{2\sqrt5}2$, which is also not totally positive.
		
		Hence we have proven that $x_1 = \frac{-\sqrt{5}+\sqrt{s}}{2}$, $x_2 = \frac{1 + \sqrt{s}}{2}$, which gives $\sum_{i \geq 3} x_i^2 = \frac{7+\sqrt5}{2}$. We already know that $x_i \in \Q(\!\sqrt5)$ for $i \geq 3$, so there must be at least three more summands since $\ell\bigl(\frac{7+\sqrt5}{2}\bigr)=3$ in $\Z\bigl[\frac{1+\sqrt5}{2}\bigr]$ by Theorem \ref{th:quadratic}. This concludes the proof.
	\end{proof}
	
	Now we solve both types of biquadratic fields containing $\sqrt2$.
	
	\begin{proposition} \label{pr:sqrt2}
		If $K=\BQ{2}{s}$ for $s \neq 3,5,7$, then $\P(\O_K)\geq 5$.
		
		In particular, depending on the value of $s$ modulo $4$, the following element has length $5$:
		\[
		\alpha = \begin{cases}
			1^2 + (1-\sqrt2)^2 + (2-\sqrt2)^2 + \Bigl(\frac12 + \sqrt2 + \frac{\sqrt{s}}{2}\Bigr)^2 + \Bigl(\frac{-1+\sqrt2-\sqrt{s}+\sqrt{2s}}{2}\Bigr)^2 & \text{if $13\neq s \equiv 1$};\\
			1^2 + \sqrt2^2 + (1-\sqrt2)^2 + \Bigl(1 + \frac{-\sqrt2 - \sqrt{2s}}{2}\Bigr)^2 + \Bigl(\frac{\sqrt2}{2}-\sqrt{s}+\frac{\sqrt{2s}}{2}\Bigr)^2 & \text{if $s \equiv 3$};\\
			1 + \sqrt2^2+(1+\sqrt{2})^2+\Bigl(\sqrt{2}+\frac{1+\sqrt{13}}{2}\Bigr)^2+\Bigl(2+\frac{1+\sqrt{13}}{2}+\frac{\sqrt{2}+\sqrt{26}}{2}\Bigr)^2 & \text{if $s = 13$}.\\
		\end{cases}
		\]
	\end{proposition}
	\begin{proof}
		The proof is analogous to that of Proposition \ref{pr:sqrt5sIs1}. One handles the cases $s \equiv 1$ and $s \equiv 3$ separately and arrives at inequalities $s \leq 30$ and $s \leq 16$, respectively; thus, to get a full contradiction, one needs to cover the cases $s=13,17,21,29$ and $s=11,15$, which is done in Lemma \ref{le:comp} (\ref{it:sqrt2}) and (\ref{it:sqrt2spec}).
	\end{proof}
	
	Finally, we turn our attention to fields containing $\sqrt3$. As a side note, remember that, unlike with $\sqrt2$ and $\sqrt5$, the lower bound $5$ is known not to always be optimal -- see Proposition \ref{pr:sqrt3shock}.
	
	\begin{proposition} \label{pr:sqrt3}
		If $K=\BQ{3}{s}$ for $s \neq 5, 7$, then $\P(\O_K)\geq 5$.
		
		In particular, the following element has length $5$:
		\[
		\alpha = \begin{cases}
			1^2 + 1^2 + (2+\sqrt3)^2 + \Bigl(\frac{1+\sqrt{s}}{2} \Bigr)^2 + \Bigl(1 + \frac{1+\sqrt{s}}{2} \Bigr)^2 & \text{if $s \equiv 1 \pmod4$};\\
			1^2 + 1^2 + (2+\sqrt3)^2 + \Bigl(\frac{\sqrt{s}+\sqrt{3s}}{2} \Bigr)^2 + \Bigl(1 + \frac{\sqrt{s}+\sqrt{3s}}{2} \Bigr)^2 & \text{if $s \equiv 2 \pmod4$};\\
			1^2 + 1^2 + (2+\sqrt3)^2 + \Bigl(\frac{\sqrt3+\sqrt{s}}{2} \Bigr)^2 + \Bigl(1 + \frac{\sqrt3+\sqrt{s}}{2} \Bigr)^2 & \text{if $s \equiv 3 \pmod4$}.\\
		\end{cases}
		\]
	\end{proposition}
	\begin{proof}
		Again, the proof is analogous to that of Proposition \ref{pr:sqrt5sIs1}. Each of the three cases must be done separately. In the first and in the third, one arrives at the inequality $s \leq 46$, while the second yields $s \leq 10$. These $5+1+6$ fields must be examined separately, see Lemma \ref{le:comp} (\ref{it:sqrt3}).
	\end{proof}
	
	It is interesting that while the used element $1^2 + 1^2 + (2+\sqrt3)^2$ is not a sum of two squares, it has another representation as a sum of three squares: $1^2 + (1+\sqrt3)^2 + (1+\sqrt3)^2$.
	
	\smallskip
	
	Now we have (almost) all the pieces needed to prove the main result of Section \ref{se:P>=5}. The structure of the proof should be clear from the structure of this section, but for clarity, we repeat it here:
	
	\begin{proof}[Proof of Theorem \ref{th:main2parts} (\ref{it:main(1)})]
		First suppose that $K$ contains none of $\sqrt2$, $\sqrt3$, $\sqrt5$, $\sqrt6$, $\sqrt7$ and $\sqrt{13}$. If $m \equiv 1 \pmod4$, we use Proposition \ref{pr:mIs1}. If $m \not\equiv 1 \pmod4$, most fields are covered in Proposition \ref{pr:mNot1Pis5}, and the exceptional cases where $s$ is equal to $m+4$, $m+8$ or $m+12$ are handled in Proposition \ref{pr:m+4etc}.
		
		On the other hand, suppose that $m \in \{2,3,5,6,7,13\}$ and $K$ is not one of the seven exceptional fields. The cases $m=2$ and $m=3$ are covered in Proposition \ref{pr:sqrt2} and Proposition \ref{pr:sqrt3}, respectively. If $m=5$ and $s \equiv 1 \pmod4$, Proposition \ref{pr:sqrt5sIs1} is used, while for $s \not\equiv 1$, Proposition \ref{pr:sqrt5sNot1} applies. Finally, for $m=6$ ($7$ or $13$, resp.) one exploits Proposition \ref{pr:sqrt6} (\ref{pr:sqrt7} or \ref{pr:sqrt13P>=6}, resp.).
	\end{proof}

	\section{Fields containing \texorpdfstring{$\sqrt5$}{root of 5}} \label{se:sqrt5}
	
	In this section, we take a closer look at fields containing $\sqrt5$. They are special since for most of them, we were able to determine the exact value of the Pythagoras number, namely $5$. (The four exceptional fields are $\BQ25$, $\BQ35$, $\BQ56$ and $\BQ57$; for the first two we expect $\P(\O_K)=3$, while for the latter two $\P(\O_K)=4$.) 
	
	In Subsection \ref{ss:strongAsymptotics}, we continue where we left off in the previous section by proving that $\P(\O_K)\geq 5$ for all but finitely many fields $K=\BQ{5}{s}$ where $s \equiv 2,3 \pmod4$. This turned out to be much more involved than the seemingly analogous cases handled in Subsection \ref{ss:235}; the method which we developed for proving Proposition \ref{pr:sqrt5GeneralStrongAsymptotics} deserves further investigation and can be employed in many other situations.
	
	Subsection \ref{ss:sqrt5upper} is untypical for this text as it is the only part where we prove an upper bound on the Pythagoras number. To achieve that, we first introduce the reader to the so-called $g$-invariants $g_{R}(n)$ of a ring $R$ and then explain how they, in general, provide upper bounds for Pythagoras numbers in field extensions, generalising a result by Kala and Yatsyna \cite{KY}. Finally, we prove that $g_{\O_F}(2)=5$ holds for $F=\Q(\!\sqrt5)$, a result published in \cite{SaJapan}.

	\subsection{Lower bound in the most difficult case}\label{ss:strongAsymptotics}
	
	The aim of this subsection is to show that with finitely many exceptions, $\P(\O_K) \geq 5$ holds for all biquadratic fields $K=\BQ{5}{s}$. Recall the notation $1<m<s<t$, which in this case means $s \neq 2,3$, $s$ is square-free, $5 \nmid s$.
	Since the case $s\equiv 1\pmod4$ was already handled in Proposition \ref{pr:sqrt5sIs1}, we achieve our goal by proving the following:
	
	\begin{proposition} \label{pr:sqrt5sNot1}
		Let $K= \BQ{5}{s}$ where $s\not\equiv 1 \pmod4$, $s \neq 6, 7$. 
		Then $\P(\O_K)\geq 5$.
		
		In particular, $\ell(\alpha)=5$ holds for
		\[
		\alpha = 1^2 + \Bigl(\frac{1+\sqrt5}{2}\Bigr)^2 + \Bigl(\frac{1+\sqrt5}{2}\Bigr)^2 + (\fss + \sss)^2 + (\css + \sss)^2.
		\]
		The symbols $\floor{x}$ and $\ceil{x}$ denote the largest integer satisfying $n\leq x$ and the smallest integer satisfying $x\leq n$, respectively.
	\end{proposition}
	\begin{proof}
		The core of the proof is contained in  Proposition \ref{pr:ourStrongAsymptotics}. Once that result is proven, two tasks remain: First, one has to take care of the fields with $s \leq 3253$ where the proposition does not apply; as usual, this was done by our computer program, see Lemma \ref{le:comp} (\ref{it:sqrt5sNot1}). \notabene{At the cost of making the proof of Proposition \ref{pr:ourStrongAsymptotics} even more technical, it is possible to improve the bounds, see Observation \ref{ob:sqrt5Improvement}. Thus we actually used the computer only to check fields with $s \leq 499$.}
		
		Second, one has to exploit that for $s \geq 3254$, Proposition \ref{pr:ourStrongAsymptotics} allows only the following type of decomposition:
		\[
		\alpha = x_1^2+x_2^2 + \beta =  \Bigl(\frac{a_1+b_1\sqrt5}{2} + \sqrt{s}\Bigr)^2 + \Bigl(\frac{a_2+b_2\sqrt5}{2} + \sqrt{s}\Bigr)^2 + \beta, \quad\text{where $\beta \in \sum \O_{\Q(\!\sqrt5)}^2$.}
		\]
		We shall show that $b_1=b_2=0$ and $\frac12 a_1 = \fss$ and $\frac12 a_2 = \css$ (or vice versa). This suffices since then $\beta = 1^2 + \bigl(\frac{1+\sqrt5}{2}\bigr)^2 + \bigl(\frac{1+\sqrt5}{2}\bigr)^2$ which has length $3$ in $\O_{\Q(\!\sqrt5)}$ (remember that by Proposition \ref{pr:ourStrongAsymptotics} we can only use summands from this subfield).
		
		To achieve this, we expand brackets:
		\[
		\Bigl(\frac{a_1+b_1\sqrt5}{2} + \sqrt{s}\Bigr)^2 + \Bigl(\frac{a_2+b_2\sqrt5}{2} + \sqrt{s}\Bigr)^2 = \underbrace{\cdots\cdots}_{\in\Q(\!\sqrt5)} + (a_1+a_2)\sss + (b_1+b_2)\sqrt{5s};
		\]
		comparing this with $\alpha = \underbrace{\cdots\cdots}_{\in\Q(\!\sqrt5)} + 2\bigl(\fss+\css\bigr)\sss$, one sees that $a_1+a_2 = 2(\fss+\css)$ and $b_1+b_2=0$.
		
		To conclude the proof, one has to compare traces. As we explain below, the assumption $\atr \alpha \geq \atr (x_1^2 + x_2^2)$ together with the equalities $a_1+a_2 = 2(\fss+\css)$ and $b_1+b_2=0$ and parity conditions $a_1\equiv b_1$ and $a_2\equiv b_2 \pmod2$ leaves only five possible choices (if we fix that either $a_1>a_2$, or $a_1=a_2$ and $b_1\geq b_2$); namely, $(a_1,a_2,b_1,b_2)$ can either take the correct value $(2\css, 2\fss, 0, 0)$, or one of the following four:
		\begin{align*}
			&\bigl(2\css+2,2\fss-2,0,0\bigr); \qquad \bigl(\css+\fss,\css+\fss,1,-1\bigr);\\
			&\bigl(\css+\fss+2,\css+\fss-2,\pm 1, \mp 1\bigr).
		\end{align*}
		For the four incorrect possibilities, one indeed gets $\beta = \alpha - (x_1^2+x_2^2) \in \Q(\!\sqrt5)$ with nonnegative trace, but in none of the cases is $\beta$ totally nonnegative. This concludes the proof.
		
		To help the reader, we explain how to compare the traces efficiently. By evaluating the traces, one gets
		\[
		2s + \fss^2 + \css^2 + 4 \geq \frac{a_1^2+a_2^2 + 5(b_1^2+b_2^2)}{4} + 2s;
		\]
		on the right-hand side we use the general identity $2x^2 + 2y^2 = (x+y)^2+(x-y)^2$ together with the known values of $a_1+a_2$ and $b_1+b_2$ to obtain
		\[
		\fss^2 + \css^2 + 4 \geq \frac{(a_1-a_2)^2 + 5(b_1-b_2)^2}{8} + \frac{(\fss + \css)^2}{2}.
		\]
		Surprisingly, this is can be rewritten (for $\sqrt{s}\notin \Z$) as $36 \geq (a_1-a_2)^2 + 5(b_1-b_2)^2$, which, together with the parity conditions, leaves only the five possibilities listed above.
	\end{proof}
	
	Before proving Proposition \ref{pr:ourStrongAsymptotics}, we start with a general result which provides very strong information about the possible decompositions of a whole family of elements; this family contains also $\alpha$. Results of this type could be proved for extensions of fields other than $\Q(\!\sqrt5)$, but in no other case were we forced to use them. Avoiding them was practical, since counterexamples of this type lead to a rather large number of exceptional fields (compare part (\ref{it:sqrt5sNot1}) of Lemma \ref{le:comp} with other parts of the same lemma).
	
	\begin{proposition} \label{pr:sqrt5GeneralStrongAsymptotics}
		Let $\alpha_0,A_1,A_2 \in \O_{\Q(\!\sqrt5)}$ and $n \in \N$. Then there exists a bound $S(n)$ such that for $s>S(n)$, $s\not\equiv 1 \pmod4$, (square-free, $5 \nmid s$), the following holds:
		
		Denote $K = \BQ{5}{s}$. Whenever the number
		\[
		\alpha = \alpha_0 + \bigl(\floor{\sqrt{s}} + A_1 + \sqrt{s}\bigr)^2 + \bigl(\floor{\sqrt{s}} + A_2 + \sqrt{s}\bigr)^2
		\]
		is represented as a sum of at most $n$ squares in $\O_K$, then exactly two of the squares belong to $(\O_{\Q(\!\sqrt5)} + \sss)^2$ and the remaining lie in $(\O_{\Q(\!\sqrt5)})^2$.
	\end{proposition}
	\begin{proof}
		First we observe that absolute values of two of the conjugates of $\alpha$ are bounded by a constant $k = k(\alpha_0,A_1,A_2)$ independent on $s$; this is clearly true for the two automorphisms of $K$ sending $\sqrt{s} \mapsto -\sqrt{s}$, since the conjugates of $\alpha_0$, $A_1$ and $A_2$ are only constants and $\floor{\sqrt{s}}-\sqrt{s}$ is bounded. From this one easily sees that if $\alpha = \sum x_i^2$ with $x_i = a_i+b_i\sqrt5+c_i\sss+d_i\sqrt{5s}$ \textbf{(where the coefficients belong to $\Q$, not necessarily $\Z$)}, then
		\begin{align*}
			|a_i-b_i\sqrt5-c_i\sss+d_i\sqrt{5s}| &< \sqrt{k},\\
			|a_i+b_i\sqrt5-c_i\sss-d_i\sqrt{5s}| &< \sqrt{k}.
		\end{align*}
		Triangle inequality then yields
		\begin{equation}\label{eq:2ineq}
			|a_i - c_i\sss| < \sqrt{k} \qquad \text{ and } \qquad |b_i\sqrt5 - d_i\sqrt{5s}| < \sqrt{k}.   
		\end{equation}
		
		We shall exploit this later. Now we need to find any constants $C,D \in \Q$ and $S_1 \in\N$ such that for $s > S_1$, we have $|c_i| \leq C$ and $|d_i| \leq D$. This is easily done by comparing traces, since $\atr\alpha$ is $4s + O(\sqrt{s})$ while $\atr x_i^2 \geq c_i^2s$ and $\atr x_i^2 \geq d_i^25s$; thus by choosing a large enough $S_1$ we can even deduce $c_i^2 \leq 4$ and $5d_i^2 \leq 4$. (Since $c_i,d_i \in \frac{\Z}{2}$, we have $C=2$ and $D=\frac12$; but the exact value of these constants is not important.)
		
		From now on consider only $s>S_1$. Our plan is to compare the asymptotic behaviour of $\atr\alpha$ and $\sum \atr x_i^2$. We already mentioned that $\atr\alpha$ can be written as $4s + x(s)$, where $|x(s)| \leq k_1 \sss + k_2$; here $k_1$ and $k_2$ are explicit constants depending only on $\alpha_0,A_1,A_2$.
		
		So it remains to analyse $\sum\atr x_i^2$, where $\atr x_i^2 = a_i^2+5b_i^2+sc_i^2+5sd_i^2$. Thanks to \eqref{eq:2ineq} we can replace the first two terms: There is a $\delta \in (-\sqrt{k},\sqrt{k})$ such that
		\[
		a_i^2 = (c_i\sss + \delta)^2 = sc_i^2 + 2\delta c_i\sss + \delta^2 = sc_i^2 + y,
		\]
		where $|y| \leq (2\cdot\sqrt{k} C)\sss + (\sqrt{k})^2$. Similarly $5b_i^2 = 5sd_i^2 + z$ where $|z| \leq (2\cdot\sqrt{k} D\sqrt{5})\sss + (\sqrt{k})^2$. All in all,
		\[
		\atr x_i^2 = (sc_i^2 + y) + (5sd_i^2 + z) + sc_i^2 + 5sd_i^2 = 2(c_i^2 + 5d_i^2)s + (y+z) = 2(c_i^2 + 5d_i^2)s + O(\sss).
		\]
		
		So far the number of squares in the decomposition played no role; what we have proved actually holds for any $x_i^2  \preccurlyeq \alpha$. Now, in the last step, the constants start depending on $n$:
		\begin{align*}
			\sum_{i=1}^n \atr x_i^2 = \sum_{i=1}^n \Bigl( 2(c_i^2 + 5d_i^2)s + O(\sss) \Bigr) &= 2\Bigl(\sum_{i=1}^n  c_i^2 + 5d_i^2\Bigr)s + n O(\sss)\\
			&= 2\Bigl(\sum_{i=1}^n  c_i^2 + 5d_i^2\Bigr)s + O(\sss).
		\end{align*}
		
		Since this also has to be equal to $\atr \alpha = 4s + O(\sqrt{s})$, we finally conclude that for $s>S(n)$ the coefficients in front of $s$ are equal:
		\[
		2 = \sum_{i=1}^n (c_i^2 + 5 d_i^2). 
		\]
		
		Since $c_i,d_i$ are either both in $\frac{\Z}{2} \setminus \Z$ or both in $\Z$, one easily sees that each nonzero summand in this sum is either $1^2+0 = 1$ or $\bigl(\frac12\bigr)^2 + 5\bigl(\frac12\bigr)^2 = \frac32$. However, to obtain $2$ as a sum of these numbers, the only choice is $1+1$. Therefore for each $i$ we have either $c_i=d_i=0$, or $|c_i|=1$, $d_i=0$, and the latter case happens for exactly two indices $i_1,i_2$. Since $x_i^2=(-x_i)^2$, we can choose $c_{i_1}=c_{i_2}=1$, which concludes the proof.
	\end{proof}
	
	The result we need to complete the proof of Proposition \ref{pr:sqrt5sNot1} is just the previous Proposition with concrete choices $\alpha_0=1 + \bigl(\frac{1+\sqrt5}{2}\bigr)^2 + \bigl(\frac{1+\sqrt5}{2}\bigr)^2 = 4+\sqrt5$, $A_1=0$, $A_2=1$, and with the effectively computed bound $S(4)$. The main reason for proving first the general version was to make the proof easier to read. Now, going through the previous proof step by step, computing the bounds explicitly (and being slightly more careful in order to get reasonably small values) yields the following:
	
	\begin{proposition} \label{pr:ourStrongAsymptotics}
		Denote $K = \BQ{5}{s}$ for $s \geq 3254$, $s\not\equiv 1 \pmod4$, (square-free, $5 \nmid s$). Then the following holds:
		
		Whenever the number
		\[
		\alpha = 1^2 + \Bigl(\frac{1+\sqrt5}{2}\Bigr)^2 + \Bigl(\frac{1+\sqrt5}{2}\Bigr)^2 + \bigl(\floor{\sqrt{s}} + \sss\bigr)^2 + \bigl(\css + \sss\bigr)^2
		\]
		is represented as a sum of four squares in $\O_K$, then exactly two of the squares belong to $(\O_{\Q(\!\sqrt5)} + \sss)^2$ and the remaining two lie in $(\O_{\Q(\!\sqrt5)})^2$.
	\end{proposition}
	\begin{proof}
		The reader has already seen the general version of this proof in Proposition \ref{pr:sqrt5GeneralStrongAsymptotics}. Thus we only include concrete values of some bounds as stepping stones.
		
		To avoid fractions, we denote $\phi = \frac{1+\sqrt5}{2}$ and $\overline{\phi} = \frac{1-\sqrt5}{2}$.
		
		If $\alpha'$ is the image of $\alpha$ in the automorphism of $K$ sending $\sqrt5 \mapsto -\sqrt5$, $\sss \mapsto -\sss$ and $\alpha''$ in the automorphism $\sqrt5 \mapsto \sqrt5$, $\sss \mapsto -\sss$, then clearly
		\[
		\alpha' = 1 + 2\overline{\phi}^2 + (\fss -\sss)^2 + (\css-\sss)^2 < 1 + 2\overline{\phi}^2 + 1 = 5 - \sqrt5
		\]
		and similarly
		\[
		\alpha'' < 2 + 2\phi^2 = 5 + \sqrt5.
		\]
		Thus
		\begin{align*}
			|a_i-b_i\sqrt5-c_i\sss+d_i\sqrt{5s}| &< \sqrt{5-\sqrt5},\\
			|a_i+b_i\sqrt5-c_i\sss-d_i\sqrt{5s}| &< \sqrt{5+\sqrt5}.
		\end{align*}
		By triangle inequality
		\begin{equation}\label{eq:2ineqVAR}
			|a_i - c_i\sss| < \kappa \qquad \text{ and } \qquad |b_i\sqrt5 - d_i\sqrt{5s}| < \kappa,    
		\end{equation}
		where $\kappa = \frac12 \Bigl(\! \sqrt{5-\sqrt5}+\sqrt{5+\sqrt5}\Bigr)$.
		
		Now we shall express the asymptotic behaviour of $\atr\alpha$. Denote $\{s\}=\sss - \fss$. Then
		\begin{align*}
			\atr\alpha &= 1 + 2\cdot\frac32 + \fss^2 + s + \css^2 + s\\
			&= 4 + 2s + (\sss-\{s\})^2 + (\sss + (1-\{s\}))^2\\
			&= 4 + 4s - 2\{s\}\sss + 2(1-\{s\})\sss + \{s\}^2 +  (1-\{s\})^2\\
			&= 4 + 4s + \lambda\sss + \mu\\
			&= 4s + \lambda\sss + (4+\mu),
		\end{align*}
		where $\lambda \in (-2,2)$ and $\mu \in (\frac12,1)$. Thus we got
		\begin{equation}\label{eq:traceAsymp}
			\mathopen|\atr\alpha - 4s\mathclose| < 2\sss + 5.
		\end{equation}
		
		We first use this inequality together with $\atr\alpha \geq s(c_i^2+5d_i^2)$ to find $S_1$ such that for $s>S_1$ we have $|c_i| \leq 2$ and $|d_i| \leq \frac12$. Suppose that one of these inequalities does not hold, i.e.\ either $|c_i| \geq \frac52$ or $|d_i|\geq 1$. Then $\atr\alpha \geq 5s$, giving $s < 2\sss + 5$, which is a contradiction already for $s \geq 12$.
		
		From now on consider only $s\geq 12$. From \eqref{eq:2ineqVAR} we know that there exists $\delta_i \in (-\kappa,\kappa)$ such that
		\[
		a_i^2 = (c_i\sss + \delta_i)^2 = sc_i^2 + 2\delta_i c_i\sss + \delta_i^2 = sc_i^2 + y_i,
		\]
		where $|y_i| \leq (2\kappa C)\sss + \kappa^2$. Similarly $5b_i^2 = 5sd_i^2 + z_i$ where $|z_i| \leq (2\kappa D\sqrt5)\sss + \kappa^2$. All in all,
		\[
		\atr x_i^2 = (sc_i^2 + y_i) + (5sd_i^2 + z_i) + sc_i^2 + 5sd_i^2 = 2(c_i^2 + 5d_i^2)s + w_i,
		\]
		where $|w_i| = |y_i+z_i| < \kappa(4+\sqrt5)\sss + 2\kappa^2$.
		
		Now, assuming that there are only four summands, and summing over all $i$, this yields
		\[
		\sum_{i=1}^4 \atr x_i^2 = 2\Bigl(\sum_{i=1}^4  c_i^2 + 5d_i^2\Bigr)s + \sum_{i=1}^4 w_i,
		\]
		hence, since the left-hand side is equal to $\atr\alpha$, we conclude
		\[
		\Bigl|\atr\alpha - 2\Bigl(\sum_{i=1}^4  c_i^2 + 5d_i^2\Bigr)s \Bigr| < 4\bigl(\kappa(4+\sqrt5)\sss + 2\kappa^2\bigr).
		\]
		
		Finally we can draw conclusions about the value of $\Sigma := \sum_{i=1}^4 (c_i^2 + 5d_i^2)$: Comparing the last inequality with \eqref{eq:traceAsymp}, we obtain
		\[
		|s(4 - 2\Sigma)| < \bigl(2 + 4\kappa(4+\sqrt5)\bigr)\sss + (5 + 8\kappa^2).
		\]
		
		One easily sees that $\Sigma$ can only take values from $\frac{\Z}{2}$, since the same is true for each of the four summands. Therefore $\Sigma \neq 2$ would mean $|s(4-2\Sigma)| \geq s$. Hence if $\Sigma \neq 2$, then the previous inequality implies
		\[
		s < \bigl(2 + \kappa(16+4\sqrt5)\bigr)\sss + (5 + 8\kappa^2) \approx 56.29\sss + 42.89.
		\]
		
		This inequality holds only for $s \leq 3253$.
		
		Thus for $s \geq 3254$, it is necessary that $\Sigma = \sum_{i=1}^4 (c_i^2 + 5d_i^2) = 2$. Drawing the conclusion that without loss of generality $c_1=c_2=1$ and the other coefficients $c_i$, $d_i$ are zero is easy (and was done in the previous proof already).
	\end{proof}
	
	There are many ways how to improve the bound $3254$ significantly. In our presentation, we preferred the simplicity of the proof over the optimality of the result. However, in order to spare computation time in proving Lemma \ref{le:comp} (\ref{it:sqrt5sNot1}), it is better to use the following stronger bound, for which we only provide a sketch of the proof:
	
	\begin{observation} \label{ob:sqrt5Improvement}
		The condition $s \geq 3254$ in Proposition \ref{pr:ourStrongAsymptotics} can be relaxed to $s \geq 500$.
		
		To see this, one must go over the corresponding proof slower and more carefully. First show independently by comparing traces that $\Sigma = \sum_{i=1}^4 (c_i^2 + 5d_i^2) \leq 4$ unless $s \leq 32$. From this inequality deduce all possible values of $(c_1, d_1, \ldots, c_4, d_4)$. Use this to observe that $\sum \bigl(|c_i| + \sqrt5 |d_i|\bigr) \leq 2 + \sqrt5$. \notabene{This in fact holds even without the restriction that there are only four summands. The reason is that $c_i^2+5d_i^2 \neq 0$ can anyway hold at most four times.} 
		This enables to avoid the rough estimates $|c_i|\leq 2$, $|d_i| \leq \frac12$ and get a stronger inequality then the one listed in the proof:
		\[
		\Bigl|\atr\alpha - 2\Bigl(\sum_{i=1}^4  c_i^2 + 5d_i^2\Bigr)s \Bigr| < 2\kappa(2+\sqrt5)\sss + 4\cdot 2\kappa^2,
		\]
		leading\notabene{In fact, by a nice trick it can be shown that $2\kappa^2$ can be replaced by exactly $5$. But it would make the proof even longer, and it yields only a slight improvement: $s < 466.4$.} directly to
		\[
		|s(4 - 2\Sigma)| < \bigl(2 + 2\kappa(2+\sqrt5)\bigr)\sss + (5 + 8\kappa^2).
		\]
		This implies $\Sigma = 2$ unless $s < 499.8$.\notabene{By the way, I'm really curious whether the \emph{whole} bound can be done independently of the number of summands. I can improve $2\kappa^2$ to $5$, but after that I must multiply it by the number of summands. Yes, I can prove that the summands outside $\Q(\!\sqrt5)$ are at most 4, but what about the others?}
	\end{observation}
	
	

	\subsection{Upper bound} \label{ss:sqrt5upper}
	
	Before stating a general result (Proposition \ref{pr:upperbound}) which provides an upper bound for Pythagoras numbers, we must introduce a generalisation of the Pythagoras number, the so-called \emph{$g$-invariants} of a ring. For $\Z$, they were introduced by Mordell, see \cite{Mo2}, their study being called \enquote{quadratic Waring's problem}. For recent results on them, see e.g.\ \cite{CI, KrY}.
	
	Let $R$ be a commutative ring. A homogeneous polynomial in $n$ variables over $R$ of degree two is called an $n$-ary \emph{quadratic form}; if it is of degree one, we call it a \emph{linear form}. A quadratic form $Q$ \emph{represents} $\alpha \in R$ if there exists a nonzero $\vec{x} = (x_1, \ldots, x_n) \in R^n$ such that $\alpha = Q(x_1, \ldots, x_n) = Q(\vec{x})$. If a form represents only totally positive elements, it is \emph{totally positive definite}; a \emph{totally positive semidefinite} form may also represent zero. The simplest totally positive definite $n$-ary quadratic form, sum of $n$ squares, is traditionally denoted by $I_n$. Thus, the fact that a number $\alpha\in R$ is a sum of, say, five squares can be rephrased as \enquote{$\alpha$ is represented by $I_5$ over $R$}. More generally, let $Q$ be an $n$-ary quadratic form and $A$ a $k$-ary quadratic form, $1 \leq k \leq n$. We say that $A$ is \emph{represented} by $Q$ (or $Q$ \emph{represents} $A$) if there exist linear forms $\L_1, \ldots, \L_n : R^k \to R$ such that
	\[
	Q\bigl(\L_1(\vec{x}), \ldots, \L_n(\vec{x}) \bigr) = A(\vec{x}).
	\]
	Clearly, $Q$ represents $\alpha\in R$ if and only if it represents the unary form $\alpha x^2$.
	Two quadratic forms are \emph{equivalent} if one can be obtained from the other by an invertible substitution.
	
	Let us now consider the set $\mathrm{Sq}_k$ of all $k$-ary quadratic forms which can be written as a sum of squares of linear forms, i.e.\ which are represented by $I_n$ for some, possibly large, $n$. Then, similarly to the definition of the Pythagoras number, we put
	\[
	g_R(k) = \min\{n : I_n \text{ represents all forms in } \mathrm{Sq}_k\}.\] If no such $n$ exists, we put $g_R(k) = \infty$; however, if $R$ is an order in a number field, then $g_R(k)$ is finite for every $k$. (For maximal orders, see \cite{Ic}; the general statement is \cite{KrY}, Corollary 1.3.)
	In particular, $g_R(1) = \P(R)$.
	
	Very few exact values of $g_R(k)$ are known. Mordell and Ko proved that $g_{\Z}(k) = k+3$ for $2 \leq k \leq 5$, see e.g.\ \cite{Ko}; much later, Kim and Oh showed $g_{\Z}(6)=10$ in \cite{KO}. In Theorem \ref{th:sasaki}, we prove $g_{\O_F}(2)=5$ for $F=\Q(\!\sqrt5)$; this was probably proven by Sasaki and published in \cite{SaJapan}. For discussion of values of $g_R(k)$ where $R$ is a global field, local ring or a nonreal maximal order, see the introduction of \cite{Ic}.
	
	The $g$-invariants provide upper bounds on the Pythagoras number. Kala and Yatsyna showed in \cite{KY}, Corollary 3.3, that $g_{\Z}(k)$ is an upper bound for $\P(\O)$ where $\O$ is any order of degree $k$; this is the upper bound $\P(\O) \leq 7$ which holds for all biquadratic orders. We exploit their idea to prove the following generalisation:
	
	\begin{proposition}\label{pr:upperbound}
		Let $R \supset S$ be commutative rings such that $R$ is a free module over $S$ of rank $k$. Then
		\[
		\P(R) \leq g_{S}(k).
		\]
	\end{proposition}
	\begin{proof}
		To simplify the notation, we assume $k=2$ (which is the case we shall need); the general proof is entirely analogous. Let $\beta_1,\beta_2 \in R$ be an integral basis of $R$ over $S$. Now take any $\alpha \in R$ and assume it is a sum of $N_\alpha \in \N$ squares; we show that it can be rewritten as a sum of $g = g_{S}(2)$ squares. We have
		\[
		\alpha = \sum_{i=1}^{N_\alpha} (a_i\beta_1 + b_i\beta_2)^2.
		\]
		Define a binary quadratic form $Q_{\alpha}$ over $S$ by
		\[
		Q_{\alpha}(X,Y) = \sum_{i=1}^{N_\alpha} (a_iX + b_iY)^2;
		\]
		since it is a sum of squares of linear forms over $S$, it can be written as
		\[
		Q_{\alpha}(X,Y) = \sum_{i=1}^{g} (c_iX + d_iY)^2.
		\]
		This is an equality of two polynomials, so it is valid after replacing $X,Y$ by elements of any commutative ring containing $S$. In particular, we can plug in the integral basis of $R$:
		\[
		\alpha = \sum_{i=1}^{N_\alpha}(a_i\beta_1 + b_i\beta_2)^2 = \sum_{i=1}^g (c_i\beta_1 + d_i\beta_2)^2.
		\]
		This is the desired representation of $\alpha$ as a sum of $g$ squares.
	\end{proof}

	\begin{remark}
		The condition that $R$ is free as an $S$-module is not necessary; it suffices that it is generated by $k$ elements. Since any torsion-free module of rank $r$ over a Dedekind domain is generated by $r+1$ elements, one gets the weaker estimate $\P(\O) \leq g_{\O_F}(r+1)$ for any field extension $L/F$ and any order $\O_F \subset \O \subset \O_L$. An interesting question is whether $r+1$ can be replaced by $r$ (as in the case where $\O_F$ is a PID).\footnote{This question was one of the starting points for the recent preprint \cite{KrY}. There, it is resolved as follows: If $\O_F$ is not a PID, it is more natural to use another version of the $g$-invariant (there denoted as $G_{\O_F}$) which is defined in terms of general quadratic lattices instead of just quadratic forms. In this setting, the inequality $\P(\O) \leq G_{\O_F}(r)$ indeed holds, and one also gets several other interesting inequalities for both versions of the invariant, vastly generalising our Proposition \ref{pr:upperbound}.}
	\end{remark}
	
	For the rest of this section, denote $F = \Q(\!\sqrt5)$. With the previous proposition, in order to obtain an upper bound $\P(\O) \leq 5$ for $\O \supset \O_F$, it remains to prove $g_{\O_F}(2)=5$. This is Theorem \ref{th:sasaki}. To prove it, we need a lemma from Sasaki's paper \cite{SaEng}:
	
	\begin{lemma}[\cite{SaEng}, Lemma 12] \label{le:sasaki}
		Let $F = \Q(\!\sqrt5)$ and let $Q$ be a binary quadratic form over $\O_F$. Then $Q$ is a sum of four squares of linear forms if and only if $Q$ is 
		\begin{enumerate}
			\item totally positive semidefinite,
			\item classical (i.e.\ the coefficient in front of $XY$ is in $2\O_F$),
			\item not equivalent to the form $G(X,Y) = 2X^2 + 2XY + 2\frac{1+\sqrt5}{2}Y^2$ over the completion $(\O_F)_{(2)}$, where $(2)$ is the dyadic place.
		\end{enumerate} 
	\end{lemma}
	Its proof is based on the fact that the genus of $I_4$ in $F$ contains only one class, so it suffices to determine which binary forms are represented locally. Let us note that the conditions of $Q$ being positive semidefinite and classical are implicit in Sasaki's formulation, and instead of $G$, he uses the form $2X^2 + 2\bigl(\frac{1+\sqrt5}{2}\bigr)XY + 2Y^2$ equivalent to $G$ over $(\O_F)_{(2)}$.
	
	\begin{remark}
		Due to 93:11 of \cite{OMeara}, the form $G$ from the previous statement can be characterised as the unique (up to equivalence) unimodular binary form over $(\O_F)_{(2)}$ which is anisotropic and takes only values from $2(\O_F)_{(2)}$.
	\end{remark}
	
	Since the following necessary result by Sasaki is, to our best knowledge, only available in Japanese in \cite{SaJapan}, we provide both its full statement and proof.

	\begin{theorem}[Sasaki]\label{th:sasaki}
		Let $F=\Q(\!\sqrt5)$. Then
		\[
		g_{\O_F}(2)=5.
		\]
	\end{theorem}
	\begin{proof}
		The easier inequality $g_{\O_F}(2) \geq 5$ can be proven either directly or using Lemma \ref{le:sasaki} by exhibiting any binary form which is a sum of squares and is equivalent to $G$ over the dyadic place. One such form\notabene{This form was rather randomly chosen; it is quite possible that even $G$ itself is a sum of squares in $\O_F$.} is $2X^2 + 2XY + \bigl(2\frac{1+\sqrt5}{2} + 16\bigr)Y^2 = (X+Y)^2 + X^2 + \bigl(2\frac{1+\sqrt5}{2} + 15\bigr)Y^2$; it is a sum of (five) squares since $2\frac{1+\sqrt5}{2} + 15$, as a totally positive element of $\O_{\Q(\!\sqrt5)}$, is a sum of (three) squares, and its equivalence to $G$ over the dyadic place follows e.g.\ from the previous remark.
		
		Now we shall take any binary quadratic form $Q(X,Y) = \sum_i \L_i^2(X,Y) = \sum_i (a_iX + b_iY)^2$ where $a_i,b_i \in \O_F$ and prove that it can be rewritten as a sum of five squares. Clearly, any form which is a sum of squares is totally positive semidefinite and classical; therefore, if $Q$ is not equivalent to $G$ over $(\O_F)_{(2)}$, even four squares suffice by Lemma \ref{le:sasaki}. Thus we may assume that $Q$ is equivalent to $G$ over $(\O_F)_{(2)}$. This means that $Q$ represents only elements of $2\O_F$.
		
		If one of $\L_i^2$ represents an element outside $2\O_F$, then so does $Q - \L_i^2$; therefore $Q - \L_i^2$ satisfies the conditions of Lemma \ref{le:sasaki} and can be represented by four squares, showing that $Q$ is a sum of five squares. It remains to show that it is absurd for all $\L_i^2$ to represent only elements of $2\O_F$: By plugging in $(X,Y) = (1,0)$ and $(0,1)$, one sees that both $a_i$ and $b_i$ lie in $2\O_F$; thus $\L_i^2$ takes only values from $4\O_F$. This contradicts the fact that $\sum \L_i^2$ represents $2$ over $(\O_F)_{(2)}$.
	\end{proof}
	
	Finally we can put the pieces together:
	
	\begin{theorem} \label{th:upperboundSqrt5}
		Let $K$ be a quadratic extension of $\Q(\!\sqrt5)$ and $\O$ any order in $K$ containing $\frac{1+\sqrt5}{2}$. Then $\P(\O)\leq 5$.
	\end{theorem}
	\begin{proof}
		By combining Proposition \ref{pr:upperbound} and Theorem \ref{th:sasaki}, one sees that $\P(R)\leq 5$ holds for any commutative ring $R$ which is a free module over $\O_{\Q(\!\sqrt5)} = \Z\bigl[\frac{1+\sqrt5}{2}\bigr]$ of rank $2$.
		
		Since the class number of $\Q(\!\sqrt5)$ is $1$, any finitely generated torsion-free module over $\Z\bigl[\frac{1+\sqrt5}{2}\bigr]$ is free. In particular, any ring which is a free $\Z$-module of rank $r$ and contains $\frac{1+\sqrt5}{2}$ is a free $\Z\bigl[\frac{1+\sqrt5}{2}\bigr]$-module of rank $\frac{r}{2}$. This concludes the proof.
	\end{proof}
	
	We reiterate that for biquadratic fields, the result is not only applicable to maximal orders; it also holds for any order of the form $\Z\bigl[\frac{1+\sqrt5}{2},\alpha\bigr]$ where $\alpha \notin \Q(\!\sqrt5)$ is an algebraic integer of degree $2$:
	
	\begin{corollary}
		For any biquadratic number field $K$ containing $\sqrt5$, there are infinitely many orders $\O\subset K$ satisfying $\P(\O) \leq 5$.
	\end{corollary}
	
	Moreover, Theorem \ref{th:upperboundSqrt5} does not only apply to biquadratic fields. For fields which are not totally real, it gives nothing new (recall that any order which is not totally real has $\P(\O) \leq 5$ by \cite{Pf}); but it can be applied to other totally real quartic orders. Notably, it includes the maximal order $\Z\Bigl[\!\sqrt{\frac{5+\sqrt5}{2}}\Bigr]$, which, along with $\Z\bigl[\sqrt2,\frac{1+\sqrt5}{2}\bigr]$, is the only totally real quartic order where the local-global principle for sums of squares is satisfied (by Scharlau's dissertation \cite{Sch2}).

	\section{Six is almost always a lower bound} \label{se:P>=6}
	In this section, we follow up the previous results by showing that $\P(\O_K)=5$ is still not the typical behaviour. We prove the second part of Theorem \ref{th:main2parts}, namely:
	
	\begin{theorem} \label{th:P>=6}
		Let $F = \Q(\!\sqrt{m})$ be a real quadratic field with $\P(\O_F) = 5$, i.e.\ $m \neq 2,3,5,6,7$. Then there are only finitely many totally real biquadratic fields $K \supset F$ with $\P(\O_K)\leq 5$.
	\end{theorem}
	\begin{proof}
		For $m \neq 13$, it is a direct consequence of Proposition \ref{pr:P>=6witnesses} which also contains the appropriate definition of the element $\alpha$ with length $6$. The case $m=13$ is handled separately (but using the same idea) in Proposition \ref{pr:sqrt13P>=6}.
	\end{proof}
	
	The proof is concluded in Subsection \ref{ss:P>=6proof}, while Subsection \ref{ss:P>=6idea} explains the underlying main idea which may be exploited in many other situations besides biquadratic fields with a given quadratic subfield.
	
	The case $m = 13$ is handled separately since $\Q(\!\sqrt{13})$ has very few elements of length $5$ (see Theorem \ref{th:quadratic}); also, the obtained Proposition \ref{pr:sqrt13P>=6} is used in Section \ref{se:P>=5}. However, it should be noted that the main idea of the proof remains the same as in the other cases.
	
	Now we provide an important part of the proof -- for each biquadratic field with $m \neq 2,3,5,6,7,13$, we define the appropriate $\alpha$ which usually has length $6$. We always have $\alpha = \alpha_0 + w^2$ with $\alpha_0$ as in Section \ref{se:P>=5}; however, to find the appropriate definition of $w$, we must distinguish $12$ families of fields. The simple method by which this terrifyingly looking case distinction was constructed is explained in the next two subsections.
	
	The list of all the mutually disjoint families, and for each of them the appropriate choice of $\alpha_0$ and $w$, follows. To be concise, we write only e.g.\ \enquote{Type (B1)} to mean \enquote{The field $K$ is of type (B1)}. The letter $q$ has the same meaning as in Subsection \ref{ss:biquadratic}, i.e.\ in (B1) it just means $q \equiv 3 \pmod4$ and in (B2,3) it means $q \equiv 1 \pmod4$.
	
	\begin{enumerate}
		\item If $m \not\equiv 1 \pmod{4}$, we put $\alpha_0 = 7 + (1+\sqrt{m})^2$ and: \label{it:(1)}
		\begin{enumerate}
			\item Type (B1): \label{it:(a)}
			\begin{enumerate}
				\item $q=m$, i.e.\ $m\equiv 3 \pmod{4}$ \label{it:(i)}
				\begin{enumerate}
					\item $s_0>3t_0$: Put $w = 1+\sqrt{s}$.\label{it:(A)}
					\item $s_0<3t_0$: Put $w = 1 + \frac{\sqrt{s}+\sqrt{t}}{2}$.
				\end{enumerate}
				\item $q=s$, i.e.\ $s\equiv 3 \pmod{4}$
				\begin{enumerate}
					\item $s_0>4t_0$: Put $w = 1+\sqrt{s}$.\label{it:1aiiA}
					\item $s_0<4t_0$: Put $w = \frac{\sqrt{m}+\sqrt{t}}{2}$.\label{it:1aiiB}
				\end{enumerate}
				\item $q=t$, i.e.\ $t\equiv 3 \pmod{4}$: Put $w = \frac{\sqrt{m}+\sqrt{s}}{2}$.\label{it:(iii)}
			\end{enumerate}
			\item Type (B2,3):
			\begin{enumerate}
				\item $q=s$, i.e.\ $s\equiv 1 \pmod{4}$: Put $w = \frac{1+\sqrt{s}}{2}$.
				\item $q=t$, i.e.\ $t\equiv 1 \pmod{4}$: Put $w = \frac{\sqrt{m}+\sqrt{s}}{2}$.
			\end{enumerate}
		\end{enumerate}
		\item If $m \equiv 1 \pmod{4}$, we put $\alpha_0 = 7 + \bigl(\frac{1+\sqrt{m}}{2}\bigr)^2$ and: \label{it:(2)}
		\begin{enumerate}
			\item Type (B2,3):
			\begin{enumerate}
				\item $s_0>3t_0$: Put $w = 1+\sqrt{s}$.
				\item $s_0<3t_0$: Put $w = 1 + \frac{\sqrt{s}+\sqrt{t}}{2}$.
			\end{enumerate}
			\item Type (B4): \label{it:(2b)}
			\begin{enumerate}
				\item $s_0>3t_0$: Put $w = \frac{1+\sqrt{s}}{2}$.
				\item $s_0<3t_0$: \label{it:(2bii)}
				\begin{enumerate}
					\item Type (B4a): Put $w = \frac{1+\sqrt{m}+\sqrt{s}+\sqrt{t}}{4}$. \label{it:(2biiA)}
					\item Type (B4b): Put $w = \frac{1+\sqrt{m}+\sqrt{s}-\sqrt{t}}{4}$.
				\end{enumerate}
			\end{enumerate}
		\end{enumerate}
	\end{enumerate}

	Let us note that these families indeed cover all possibilities, since $s_0=4t_0$ cannot happen for $s_0$ square-free and $s_0=3t_0$ is possible only for $s_0=3$, $t_0=1$, leading to $m=3$ which violates the assumption $\P(\O_{\Q(\!\sqrt{m})})=5$.
	
	We conclude this part by formulating what remains to be proven.
	
	\begin{proposition}\label{pr:P>=6witnesses}
		Let $F=\Q(\!\sqrt{m})$ be a given field as in Theorem \ref{th:P>=6}, with $m\neq 13$, and for each totally real biquadratic field $K \supset F$ define $\alpha = \alpha_0+w^2$ as above. Then, with finitely many exceptions, $\ell(\alpha)=6$ in $\O_K$.
	\end{proposition}
	\begin{proof}
		For a given $m$, there are only finitely many biquadratic fields $K$ where $m$ is not the smallest of the three square-free integers whose square roots lie in $K$; disregarding them, we may assume that the usual Convention \ref{no:mst} holds. Clearly, every $m$ corresponds to only finitely many choices of $(t_0,s_0)$.
		
		After realising this, it suffices to prove twelve separate lemmata, where for example the first of them, with code name \eqref{it:(A)}, claims: \enquote{Let $s_0>t_0\geq 1$ be two coprime square-free numbers such that $s_0t_0 \equiv 3 \pmod4$ and $s_0>3t_0$. Put $\alpha_0=7+(1+\sqrt{s_0t_0})^2$. If we take any large enough square-free $m_0$ coprime with $s_0t_0$, and then denote $K=\BQ{s_0t_0}{m_0t_0}$ and put $w=1+\sqrt{m_0t_0}$, then $\ell(\alpha_0+w^2) = 6$ in $\O_K$.}
		
		Another of these twelve statements, Lemma \ref{le:oneOf12}, is formulated and proven in Subsection \ref{ss:P>=6proof}. The others are omitted, since all the twelve proofs are very similar, and after understanding one of them, it is not difficult to write any of the others.
	\end{proof}

	\subsection{The idea} \label{ss:P>=6idea}
	
	Here we suggest a general strategy which applies in many situations when one examines $\P(\O_K)$ where $K$ runs over extensions of a given field $F$ with an already known Pythagoras number $\P(\O_F)$. This strategy was applied to obtain the appropriate definition of $\alpha$ for the situation of Theorem \ref{th:P>=6}, but it can be tried for any field, not just the biquadratic ones. The aim is to find conditions under which $\P(\O_K) \geq \P(\O_F) + 1$ can be proven: For example, $\P(\Z)=4$, and indeed, $\P(\O_{\Q(\!\sqrt{n})}) = 5$ holds for almost all positive $n$. Or, as we have already seen in Subsection \ref{ss:67}: For $F = \Q(\!\sqrt6)$ or $\Q(\!\sqrt7)$, where $\P(\O_F)=4$, almost all biquadratic fields $K \supset F$ have $\P(\O_K) \geq 5$.
	
	The way of constructing the \enquote{witness} $\alpha \in \O_K$ which probably has length $P(\O_F)+1$, follows: First, one takes $\alpha_0 \in \O_F$ of length $\P(\O_F)$ in $\O_F$. (It is convenient to use one with the minimal trace, but any $\alpha_0$ can work.) It is reasonable to hope that in most extensions $K \supset F$, $\alpha_0$ will also require $\P(\O_F)$ squares. For example, the only real quadratic fields where $7$ can be written as a sum of less than four integral squares are those generated by $\sqrt2$, $\sqrt3$, $\sqrt5$, $\sqrt6$, $\sqrt7$ and $\sqrt{13}$; and a similar statement for our situation where $F$ is quadratic and $K$ biquadratic was proven in Subsections \ref{ss:P>=5mEquiv1} and \ref{ss:P>=5mEquiv23}. However, we need to find an element with length $\P(\O_F)+1$. This is often achieved by choosing $w^2 \in \O_K \setminus \O_F$ where the trace of $w^2$ is the minimal possible among such elements. Then one defines $\alpha = \alpha_0 + w^2$ and hopes to prove that any representation of $\alpha$ as a sum of squares uses $w^2$ as one of the summands. This has a good chance of success if $\Tr \alpha_0$ is much smaller than $\Tr w^2$: In such a case, $w^2$ will be one of the very few squares totally smaller or equal to $\alpha$ and not belonging to $\O_F$.
	
	This strategy was very successful in Peters' examination of quadratic fields as extensions of $\Z$: The choice $\alpha_0 = 7$ and $w = (1+\sqrt{m})$ or $w = \bigl(\frac{1+\sqrt{m}}{2}\bigr)$, see Theorem \ref{th:quadratic}, led to an element of length $5$ in all but six maximal orders, and the analogy for non-maximal orders succeeded every time except of $\Z[\sqrt5]$, see also Observation \ref{ob:betterQuadratic}.
	
	Let us note, however, that there is no guarantee that this strategy will indeed work -- as we have seen in Subsection \ref{ss:m+4,m+8,m+12}, there are infinitely many biquadratic fields where even the weaker statement $\ell(\alpha_0)=5$ fails. This is also why the proof in the next subsection had to be carefully divided into twelve different branches depending on the integral basis and on $m$, and why it was necessary to perform the proof separately for each of these twelve families. The most troublesome situation occurs if the choice of $w^2$ is not unique: In such a case it might be difficult to show that the rejected candidates for $w^2$ cannot be used in the decomposition of $\alpha$ instead of $w^2$. But the strategy turned out to be successful in all twelve families. 
	
	
	We conclude this subsection by illustrating how it applies in the exceptional case $m=13$.
	
	\begin{proposition}\label{pr:sqrt13P>=6}
		Let $K=\BQ{13}{s}$. Then $\P(\O_K)\geq 6$.
		
		In particular, $\ell(\alpha)=6$ holds for $\alpha = \alpha_0 + w^2$, where $\alpha_0 = 3+\bigl(\frac{1+\sqrt{13}}{2}\bigr)^2+\bigl(1+\frac{1+\sqrt{13}}{2}\bigr)^2 = 12 + 2\sqrt{13}$ and
		\[
		w = \begin{cases}
			\frac{1+\sqrt{s}}{2} & \text{if $s \equiv 1 \pmod4$},\\
			1+\sqrt{s} & \text{if $s \equiv 2,3 \pmod4$}.
		\end{cases}
		\]
	\end{proposition}
	\begin{proof}
		The proof is an incarnation of the above explained strategy: Since $\ell(\alpha_0)=5$ in $\O_{\Q(\!\sqrt{13})}$, see Theorem \ref{th:quadratic}, it suffices to show by trace considerations that $w^2$ must be one of the summands and all the others must belong to $\Q(\!\sqrt{13})$. See also the analogous proofs in Subsection \ref{ss:67}.
		
		More specifically: If $s\not\equiv 1$ and we assume a decomposition of $\alpha$ which is not of the form $w^2 + \cdots$ where the remaining summands are squares in $\Q(\!\sqrt{13})$, we arrive at $s \leq 13$, an immediate contradiction. For $s\equiv 1$, the obtained inequality is $s \leq 49$, which left only six fields; these were solved in Lemma \ref{le:comp} (\ref{it:sqrt13P>=6}).
	\end{proof}
	
	Note that the statement of the proposition actually holds even for fields $\BQ{13}{n}$ for $n=6,7,10,11$, see Lemma \ref{le:comp} (\ref{it:sqrt13=s}); the element $\alpha_0 + (1+\sqrt{n})^2$ has length $6$. This is interesting, especially for fields containing $\sqrt{6}$ and $\sqrt{7}$, for which we in general know only $\P(\O_K)\geq 5$; it is further evidence that fields with Pythagoras number less than $6$ are very rare, see Conjecture \ref{co:conjecture}.

	\subsection{The proof} \label{ss:P>=6proof}
	
	
	In this subsection we provide the proof of Theorem \ref{th:P>=6}. The general idea was explained at length in the previous subsection: One starts with $\alpha_0 \in \Q(\!\sqrt{m})$ which has the minimal trace among elements of length $5$, namely $\alpha_0 = 7 + \bigl(\frac{1+\sqrt{m}}{2}\bigr)^2$ if $m\equiv 1 \pmod4$ and $7 + (1+\sqrt{m})^2$ otherwise. Then one takes $w^2 \in \O_K \setminus \Q(\!\sqrt{m})$ which has minimal trace, to define $\alpha = \alpha_0 + w^2$, as the table above Proposition \ref{pr:P>=6witnesses} shows. This $w$ depends not only on the integral basis and on whether $m=q$ or $m \neq q$, but sometimes also on (in)validity of inequalities like $3s > t$.
	
	To solve this last problem, one had to consider all the finitely many decompositions $m = s_0t_0$ for the given $m$. Clearly, if we prove that for each fixed choice of $t_0$ and $s_0$ there are only finitely many $m_0$ such that the biquadratic field $K$ given by $m=t_0s_0, s=t_0m_0, t=s_0m_0$ has $\P(\O_K)\leq 5$, then we are done.
	


	Since the proofs are really the same in all twelve families, we provide proof only in one of them. We have chosen the most difficult branch where some small additional tricks were needed.
	
	\begin{lemma}[Branch \eqref{it:(2biiA)}] \label{le:oneOf12}
		If the field $K=\BQ{m}{s}$ is of type (B4a), $m \neq 5,13$ and $s_0<3t_0$ holds, and $m_0$ is large enough,\notabene{If we needed a rough estimates, due to congruences we have $3t_0-s_0 \geq 2$, hence $m_0 > \frac{117+5m}{2}$ definitely suffices. If one follows this idea further, one sees that this condition is violated only $O(m)$-times. And for each $m$, one has to consider the number of decompositions $m=s_0t_0$, which is half of the number of divisors of $m$.) $\ldots$ This could lead to showing that not only is there only finitely many exceptions for each $m$, but that they are actually only $O(m\log m)$ or whatever. Interesting, but no time and space for it.} then the following element has length $6$:
		\begin{align*}
			\alpha &= \alpha_0 + w^2 = 7 + \frac{1+m}{4} + \frac{\sqrt{m}}{2} + \Bigl(\frac{1+\sqrt{m}+\sqrt{s}+\sqrt{t}}{4}\Bigr)^2\\
			&= \Bigl(7 + \frac{5+5m+s+t}{16}\Bigr) + \frac{5+m_0}{8}\sqrt{m} + \frac{1+s_0}{8}\sqrt{s} + \frac{1+t_0}{8}\sqrt{t}.
		\end{align*}
	\end{lemma}
	\begin{remark}
		In this family, the condition is $m_0 > \frac{117+5m}{3t_0-s_0}$, as one sees from the proof.
	\end{remark}
	\begin{proof}
		Suppose $\alpha = \sum x_i^2$. Since $\alpha_0$ cannot be written as a sum of less than five squares in $\O_{\Q(\!\sqrt{m})}$, it suffices to prove that one of the $x_i^2$ must be equal to $w^2$ and all other $x_i \in \Q(\!\sqrt{m})$.
		
		Write $x_i = \frac{a_i+b_i\sqrt{m}+c_i\sqrt{s}+d_i\sqrt{t}}{4}$; since $K$ is of type (B4a), we see that $a_i,b_i,c_i,d_i$ are either all even or all odd, and $a_i+b_i+c_i+d_i \equiv 0 \pmod4$. Since $\alpha\notin \Q(\!\sqrt{m})$, at least one of the summands is not in $\Q(\!\sqrt{m})$, hence at least one of all the coefficients $c_i,d_i$ is nonzero. Due to the parity condition, either $\max\{\sum c_i^2,\sum d_i^2\} \geq 4$, or $\sum c_i^2 \geq 1$ and $\sum d_i^2 \geq 1$.
		
		Comparing traces of $\alpha$ and $\sum x_i^2$, one gets
		\[
		7 + \frac{5+5m+s+t}{16} = \frac1{16}\Bigl(\sum a_i^2 + m\sum b_i^2 + s\sum c_i^2 + t\sum d_i^2\Bigr) \geq \frac{1}{16}\Bigl(s\sum c_i^2 + t\sum d_i^2\Bigr),
		\]
		which can be rewritten as
		\[
		\frac{117+5m}{16} + m_0\frac{t_0+s_0}{16} \geq m_0 \frac{1}{16}(t_0\sum c_i^2 + s_0\sum d_i^2).
		\]
		If $(t_0\sum c_i^2 + s_0\sum d_i^2)>t_0 + s_0$, then by choosing $m_0$ large enough we obtain a contradiction. Therefore the expression must be at most $t_0+s_0$. Hence $\sum c_i^2 \geq 4$ or $\sum d_i^2 \geq 4$ is impossible because of our assumption on $s_0$ and $t_0$, since then $(t_0\sum c_i^2 + s_0\sum d_i^2) \geq 4t_0 > t_0 + s_0$. Thus necessarily $\sum c_i^2 = \sum d_i^2 = 1$. Therefore we may assume $|c_1|=|d_1|=1$ and $c_i,d_i=0$ otherwise.
		
		We have just proven $x_i\in\Q(\!\sqrt{m})$ for $i\neq 1$. Since $(-x_1)^2 = x_1^2$, we may assume $c_1=1$; hence there are two possibilities: Either $x_1 = \frac{a_1+b_1\sqrt{m}+\sqrt{s}+\sqrt{t}}{4}$, or $x_1 = \frac{a_1+b_1\sqrt{m}+\sqrt{s}-\sqrt{t}}{4}$. In both cases, we will use the fact that $\alpha - x_1^2 \in \Q(\!\sqrt{m})$ to determine the values of $a_1$ and $b_1$.
		
		In the first case,
		\[
		x_1^2 = \Bigl(\frac{a_1+b_1\sqrt{m}+\sqrt{s}+\sqrt{t}}{4}\Bigr)^2 \in \Q(\!\sqrt{m}) + \frac{a_1+b_1s_0}{8}\sqrt{s} + \frac{a_1+b_1t_0}{8}\sqrt{t};
		\]
		by comparing the coefficients in front of $\sqrt{s}$ and $\sqrt{t}$ with those of $\alpha$, we obtain two equations $a_1+b_1s_0=1+s_0$ and $a_1+b_1t_0=1+t_0$. Since $t_0\neq s_0$, there is a unique solution: $a_1=b_1=1$. Therefore we indeed arrived at $x_1 = \frac{1+\sqrt{m}+\sqrt{s}+\sqrt{t}}{4}$, which is what we needed.
		
		It remains to be shown that the other case is impossible. Again we write
		\[
		x_1^2 = \Bigl(\frac{a_1+b_1\sqrt{m}+\sqrt{s}-\sqrt{t}}{4}\Bigr)^2 \in \Q(\!\sqrt{m}) + \frac{a_1-b_1s_0}{8}\sqrt{s} + \frac{-a_1+b_1t_0}{8}\sqrt{t};
		\]
		this time we obtain the equations $a_1-b_1s_0=1+s_0$ and $-a_1+b_1t_0=1+t_0$. Again, the solution is unique: $a_1 = \frac{t_0+s_0+2s_0t_0}{t_0-s_0}$, $b_1 = \frac{t_0+s_0+2}{t_0-s_0}$. However, in order that $x_1 \in \O_K$, it is necessary that $a_1$ and $b_1$ are odd integers. But it is easy to check that if $s_0\equiv t_0 \equiv 1 \pmod4$, which are the conditions of (B4a), then $b_1 = \frac{t_0+s_0+2}{t_0-s_0} = 1 + 2\frac{s_0+1}{t_0-s_0}$ is either even, or not an integer. This concludes the proof.
	\end{proof}
	
	\subsection{A conjecture} \label{ss:conjecture}
	We suspect that it is in fact possible to prove a stronger result, listed already in the Introduction as the first part of Conjecture \ref{co:conjecture}. However, we did not try proving it in order not to make this paper longer than it already is. Moreover, the potential proof would not only have many branches, but probably some of them would require new ideas. As the last part of this paper, let us repeat the Conjecture \ref{co:conjecture} and list some evidence towards it:
	
	\begin{conjecture*}
		Let $K$ be a totally real biquadratic field.
		\begin{enumerate}
			\item If $K$ contains none of $\sqrt2$ and $\sqrt5$, then $\P(\O_K) \geq 6$ holds with finitely many exceptions.
			\item If $K$ contains $\sqrt2$ or $\sqrt5$, then $\P(\O_K) \leq 5$.
			\item The inequality $\P(\O_K)<5$ holds precisely for the following seven fields:\\ For $K = \BQ23, \BQ25, \BQ35$, where $\P(\O_K)=3$, and\\ for $K =\BQ27, \BQ37, \BQ56, \BQ57$, where $\P(\O_K)=4$.
		\end{enumerate}
	\end{conjecture*}
	Part (\ref{it:con2}) was already proven for $\sqrt5$; for $\sqrt2$, it is an observation based on computer calculations: Our program never found an element with length $6$ or $7$. Of course we do not have a proof for even one concrete field, let alone for all extensions of $F = \Q(\!\sqrt2)$. We suggest that in analogy with Theorem \ref{th:sasaki}, one should try to prove $g_{\O_{\Q(\sqrt2)}}(2)=5$. This seems feasible, since the crucial point in the proof of Lemma \ref{le:sasaki}, namely that the genus of $I_4$ contains only one class, remains valid for $\Q(\!\sqrt2)$.\footnote{This strategy of the proof was successfully carried out by He and Hu in \cite{HH} before the present paper was published. Thus, part (2) of the Conjecture in now in fact a theorem.}
	
	\smallskip
	
	Let us turn to (\ref{it:con3}). The fact that these seven are the only possible exceptional fields is proven as Theorem \ref{th:main2parts} (\ref{it:main(1)}), and an element of the suspected maximal length in each of them was already given in Proposition \ref{pr:comp_examples}. Thus, it only remains to prove a stronger upper bound for these seven fields. Proposition \ref{pr:comp_examples} also contains some numerical evidence -- namely, the conjecture is valid for elements with trace up to $500$.
	
	A further point in favour of this part of the conjecture is that it holds locally -- by the result \cite{Sch2}, Kapitel 0, Lemma 1, already mentioned in the Introduction, four integral squares are sufficient in all completions $K_{\mathfrak{p}}$, and three squares suffice unless $\mathfrak{p}$ is dyadic and $[K_{\mathfrak{p}}:\Q_2]$ is odd. Since biquadratic fields are Galois extensions of $\Q$ of even degree, one can show that such a problematic dyadic place can exist only for fields of type (B4), where $2$ does not ramify. None of the seven listed fields are of this type.
	
	\smallskip
	
	Regarding part (\ref{it:con1}), again it is based on computer experiments. There would be two separate tasks in proving it. First, one has to handle the cases not included in Theorem \ref{th:P>=6}, namely fields containing $\sqrt3$, $\sqrt6$ or $\sqrt7$. Here we observed from our data that even for quite small values of $s$, the corresponding $\O_K$ usually contained an element of length $6$. (For instance, for $K \ni \sqrt3$, Proposition \ref{pr:sqrt3shock} suggests that $\P(\O_K)\geq 6$ holds for all $s\geq 26$ as well as for $s=17$ and $s=22$, leaving only nine exceptions.) To prove this, one would either have to find in the computed data a family of such elements (as in Subsection \ref{ss:235} for length $5$ and $K \ni \sqrt2, \sqrt3$ or $\sqrt5$); or alternatively, if no such family could be found, one could prove an analogy of Proposition \ref{pr:sqrt5GeneralStrongAsymptotics} and construct a family of counterexamples in this way.
	
	And the second task would be to prove a stronger version of Theorem \ref{th:P>=6}: If $m$ is chosen large enough, then there are no exceptional fields. It would be ideal to show that the corresponding $\alpha = \alpha_0 + w^2$ \emph{always} requires six squares. However, this is not true; at least some infinite families of exceptional fields \emph{can} be found: $s=m+4$, $s=m+8$ and $s=m+12$. Possibly they can be handled similarly as in Subsection \ref{ss:m+4,m+8,m+12}. However, it is clear that solving all cases would be a tedious task requiring both patience and clever ideas.
	
	A further indication that the inequality $\P(\O_K)\geq 6$ indeed holds for all but finitely many fields (excluding $m=2$ or $5$) is that it is almost true in the case of $m,s$ coprime (see Proposition \ref{pr:coprime}). As a final piece of evidence we list a result which was shown to us by M.\ Tinková (including a proof) and will be part of an upcoming paper of hers. Since it contains no phrase \enquote{up to finitely many exceptions}, it is a stronger version of Theorem \ref{th:P>=6} for two of the twelve families, namely those coded by \eqref{it:(A)} and \eqref{it:1aiiA}, as well as an alternative for a third family, \eqref{it:1aiiB}.
	
	\begin{theorem}[Tinková]\label{th:MagdaNotCoprime}
		Let $K=\BQ{p}{q}$ where $p,q$ are square-free positive integers such that $p\equiv2$ and $q\equiv3 \pmod4$. Moreover, let $q_0> r_0\geq 3$ and $p_0 > 3r_0$. Then $\P(\O_K)\geq 6$.
		
		In particular, $\alpha=7+(1+\sqrt{p})^2+(1+\sqrt{q})^2$ has length $6$.
	\end{theorem}


	\section*{Acknowledgements}
	
	We thank Pavlo Yatsyna and V{\'\i}t{\v e}zslav Kala for suggesting such a rich and interesting topic and for all their support throughout this project. We are very grateful to Magdal{\'e}na Tinkov{\'a} who shared a lot of her experience and significantly contributed to the contents of Section \ref{se:coprime}. It also helped us that Rudolf Scharlau kindly provided us with a printed version of his dissertation. We also acknowledge the friendly reaction of both Rudolf Scharlau and Hideyo Sasaki to our e-mail, and the helpful remarks of the anonymous referee. Last but not least, Barbora Hudcov{\'a} deserves our gratitude for her help with the Japanese language.
	
	J.\ K.\ and E.\ S.\ were supported by project PRIMUS/20/SCI/002 from Charles University. J.\ K.\ was further partially supported by Czech Science Foundation GA\v{C}R, grant 21-00420M, by projects UNCE/SCI/022 and GA UK No.\ 742120 from Charles University, and by SVV-2020-260589.

	\addtocontents{toc}{~\vspace{2em}{}\par} 
	
\end{document}